\documentclass[reqno, 11pt,letter]{amsart}
\usepackage{graphicx,  enumitem, amsmath, amssymb,amsthm,hyperref,amscd,xcolor}
\usepackage{mathabx,manfnt,braket,mathrsfs, comment,mathtools}
\usepackage{tikz}
\usepackage{tikz-cd}
\usepackage{cancel}
\usetikzlibrary{arrows, matrix,decorations.pathmorphing}
\usepackage[utf8]{inputenc}


\newtheorem{thm}{Theorem}[section]
\newtheorem{prop}{Proposition}[section]

\newtheorem{lem}[prop]{Lemma}
\newtheorem{cor}[prop]{Corollary}
\theoremstyle{definition}

\newtheorem*{rem}{Remark}

\numberwithin{equation}{section}

\newcommand{\img}{\operatorname{img}}
\DeclareMathOperator{\Red}{\sf{Red}}
\DeclareMathOperator{\Ind}{\sf{Ind}}

\newcommand{\norm}[1]{\left\Vert#1\right\Vert}
\newcommand{\abs}[1]{\left\vert#1\right\vert}
\newcommand{\cx}{{\mathbb{C}}}
\newcommand{\rl}{{\mathbb{R}}}

\newcommand{\zs}{\mathscr{Z}}

\newcommand{\cls}[2]{\left[{#1}\right]_{{#2}}}

\newcommand{\ol}{\overline}

\newcommand{\dbar}{\ol\partial}
\newcommand{\wt}{\widetilde}
\newcommand{\wh}{\widehat}

\newcommand{\ea}{\mathcal{E}^{\mathrm{A}}}
\newcommand{\eh}{\mathcal{E}^{\mathrm{H}}}

\newcommand{\rhoa}{\rho^{\mathrm{A}}}
\newcommand{\rhoh}{\rho^{\mathrm{H}}}

\newcommand\ipr[1]{\left\langle #1 \right\rangle}

\DeclareMathOperator{\dom}{Dom}
\DeclareMathOperator{\range}{Range}

\newcommand{\db}{\overline{\mathfrak{d}}}

\title[Exact sequences]{Exact sequences and  estimates for the $\overline{\partial}$-problem}
\author{Debraj Chakrabarti}
\address{Department of Mathematics, Central Michigan University, Mt. Pleasant,  MI 48859,  USA}
\email{chakr2d@cmich.edu}
\author{Phillip S. Harrington}
\address{SCEN 309, 1 University of Arkansas, Fayetteville, AR 72701, USA}
\email{psharrin@uark.edu}

\thanks{Debraj Chakrabarti was partially supported by NSF grant  DMS-1600371.}
\subjclass[2010]{32W05}
\begin{document}
\maketitle
\begin{abstract}We study Sobolev estimates for solutions of the inhomogenous Cauchy-Riemann equations on annuli in $\cx^n$, by constructing exact sequences relating the Dolbeault cohomology
of the annulus  with respect to Sobolev spaces of forms with those of the envelope and the hole. We also obtain solutions  with prescibed support and estimates in Sobolev spaces using our method.
\end{abstract}

%\tableofcontents % remove this later.
\section{Introduction}
{The theory of $L^2$ and Sobolev estimates for the solutions to the $\dbar$-problem (i.e., the inhomogeneous Cauchy-Riemann equations) on domains in
$\cx^n$ is a classical topic in complex analysis.
Quite naturally, the main focus has been on pseudoconvex domains, where the positivity needed to
establish the required a priori estimates is available.  Some of these  results can be extended to certain non-pseudoconvex domains  satisfying weaker positivity conditions such as  $Z(q)$ (also called $a_q$ or $A_q$) or weak $q$-convexity (see \cite{hormander1965,follandkohn,ho1991}).}

Another fruitful approach to the study of the $\dbar$-problem has been to concentrate on the simplest class of non-pseudoconvex domains, the \emph{annuli}.
	By an \emph{annulus}, we mean a bounded domain $\Omega\subset \cx^n$ which can be represented in the form
	\begin{equation}
	\label{eq-annulusdef}
	\Omega= \Omega_1 \setminus\ol{ \Omega_2}
	\end{equation}
	where $\Omega_1$ and $\Omega_2$ are bounded  open sets in $\cx^n$, with $\ol{\Omega}_2\subset \Omega_1$, and $\Omega_1$  connected.
	We say that
	$\Omega_1$ is the \emph{envelope} and $\Omega_2$ the \emph{hole} of the annulus $\Omega$, where it is not assumed that the hole
	$\Omega_2$  is connected. In \cite{shawannulus1985}, Shaw considered the question of solvability and regularity of the $\dbar$-problem on an annulus where the
	envelope and the hole were both smoothly bounded pseudoconvex domains, using an adaptation of the weighted $\dbar$-Neumann technique of Kohn (\cite{kohn73}).
	Subsequently, it was realized that one could relate the function theory on the annulus with those of the hole and annulus and thus avoid solving the $\dbar$-equations
	directly on the annulus  (see \cite{shawclosed, meichiharmonic, lishaw, meichichristine, ChShTh, FLS} and subsequent work on this theme).  The goal of this paper is to understand and extend this method
	using  elementary homological algebra and  functional analysis of non-Hausdorff topological vector spaces (TVS). A consequence of our investigations is the following characterization of annuli in which the $\dbar$-problem can be solved
	with $L^2$-estimates:
	\begin{cor}[to Theorem~\ref{thm-short} below] \label{cor-intro} Let $n$ be a positive integer and $0\leq p,q \leq n$.
	The following are equivalent for an annulus $\Omega= \Omega_1 \setminus\ol{ \Omega_2}$ in $\cx^n$:
	\begin{enumerate}
		\item $H^{p,q}_{L^2}(\Omega)=0$
		\item $H^{p,q}_{L^2}(\Omega_1)=0$ and $H^{p,q+1}_{c,W^{-1}}(\Omega_2)=0$.
	\end{enumerate}
\end{cor}
Here $H^{p,q}_{L^2}(D)$ denotes the well-known $L^2$-Dolbeault cohomology of the domain $D$, the obstruction to solving the $\dbar$-problem with $L^2$-estimates.
The statement $H^{p,q+1}_{c,W^{-1}}(\Omega_2)=0$ means the following: if $f$ is a $(p,q+1)$-current on $\cx^n$ with coefficients in the $L^2$-Sobolev space $W^{-1}(\cx^n)$ and support in $\ol{\Omega_2}$ such that $\dbar f=0$, then there is a $(p,q)$-current $u$, again with coefficients in the $L^2$-Sobolev space $W^{-1}(\cx^n)$ and again with support in $\ol{\Omega_2}$ such that $\dbar u=f$; i.e., the $\dbar$-problem
can be solved with prescribed support on $\ol{\Omega_2}$ and estimates in the  $W^{-1}$-Sobolev norm.

As Corollary~\ref{cor-intro} shows, to understand the $L^2$ solvability of the $\dbar$-problem on non-pseudoconvex domains  we are led to the consideration of
Sobolev cohomologies, including cohomologies with prescribed support. Two important types of such cohomologies in this paper
 are the maximal $W^s$-cohomology $H^{p,q}_{W^s}(D)$ of a domain, measuring the obstruction to solving the $\dbar$-problem
with estimates in the $L^2$-Sobolev space $W^s$, and the minimal $W^s$-cohomology $H^{p,q}_{c, W^s}(D)$, encountered in Corollary~\ref{cor-intro} for $s=-1$, which measures the solvability of the $\dbar$-problem
with prescribed support. These are the Sobolev analogs of the well-known maximal and minimal $L^2$-realizations of the $\dbar$-operator (see \cite{serre}). The precise definitions may be found in
Section~\ref{sec-realizations} below, but at this point we want to emphasize that in view of applications like Corollary~\ref{cor-intro}, we allow Sobolev spaces of both positive and negative orders.

While solvability of the $\dbar$-problem with estimates in a certain norm corresponds to the vanishing of the corresponding cohomology groups, for
non-pseudoconvex domains, the existence of estimates correspond to whether the groups are Hausdorff, or equivalently, whether the $\dbar$-operator has closed range.  For example, the group $H^{p,q}_{W^s}(D)$ is defined to the quotient topological  vector space $Z^{p,q}_{W^s}(D)/B^{p,q}_{W^s}(D)$, where
$Z^{p,q}_{W^s}(D)$ is the space of $\dbar$-closed  $(p,q)$-forms with coefficients in $W^s(D)$, and $B^{p,q}_{W^s}(D)$ is the subspace of $\dbar$-exact forms.
The space $Z^{p,q}_{W^s}(D)$ is an inner product space (as a subspace of $W^s_{p,q}(D)$, the space of $(p,q)$-forms with coefficients in $W^s(D)$). We endow the group
$H^{p,q}_{W^s}(D)$ with the quotient topology, which is Hausdorff if and only if $B^{p,q}_{W^s}(D)$ is closed in $Z^{p,q}_{W^s}(D)$.

Such non-Hausdorff topologies have long been understood as being an important feature of Dolbeault cohomology. In our case, we have the additional feature that
our cohomologies are quotients of \emph{inner product spaces}. As a result, the cohomology groups of this paper have the structure of \emph{semi-inner-product} spaces,
i.e., there is a sesquilinear form
which is non-negative definite, and compatible with the topology (see Section~\ref{sec-sip} below).
 Further, a semi-inner product space $X$, like other not-necessarily-Hausdorff topological vector spaces, has a topological decomposition as a direct sum (see Proposition~\ref{prop-structure}):
 \begin{equation}
 \label{eq-direct} X\cong \Red(X) \oplus \Ind(X),
 \end{equation}
{where we define the \emph{indiscrete part} of $X$ to be the subspace}
\begin{equation}
\label{eq-indiscrete}
{\Ind(X) =\{x\in X: \norm{x}=0\},}
\end{equation}
{which is easily verified to be a closed subspace of $X$ which has the indiscrete topology, i.e., the only nonempty open subset of $\Ind(X)$ is $\Ind(X)$ itself.  We define the \emph{reduced form} of $X$ to be the quotient  space}
\begin{equation}
\label{eq-reduced}
{\Red(X) =X/\Ind(X),}
\end{equation}
{which is a \emph{normed}  (and therefore Hausdorff)  space.}
In an appendix at the end of the paper, we sketch the basic facts
about  not-necessarily-Hausdorff topological vector spaces which arise as cohomology groups.

One of our main results is:

\begin{thm}
	\label{thm-short}
	Let $\Omega=\Omega_1\setminus \ol{\Omega}_2\subset \cx^n$  be a bounded annulus. Let $s$ be an integer, $0\leq q\leq n$, and $0\leq p\leq n$. The following sequence of semi-inner-product spaces
	and continuous {linear} maps is exact:
	\begin{equation}
	\label{eq-short-main}
	{0\to H^{p,q}_{W^{s+1}}(\Omega_1) \xrightarrow{{R}^{p,q}_*}H^{p,q}_{W^{s+1}}(\Omega)
		\xrightarrow{\lambda^{p,q}}H^{p,q+1}_{c,W^s}(\Omega_2)\to 0},
	\end{equation}
	where $R^{p,q}_*$  is the restriction map on cohomology, and $\lambda^{p,q}$ is defined as:
	\begin{equation}
	\label{eq-lambdadef}
	\lambda^{p,q}\left([f]_{H^{p,q}_{W^{s+1}(\Omega)}}\right)= \left[\left.\left(\dbar E f\right)\right\vert_{\Omega_2}\right]_{H^{p,q}_{c, W^s}(\Omega_2)} \quad \text{ for } 	f\in Z^{p,q}_{W^{s+1}(\Omega)},
	\end{equation}
	where $	E: W^{s+1}(\ol{\Omega})\to W^{s+1}(\cx^n)$ is an extension map acting coefficientwise (see \eqref{eq-extension} below).
\end{thm}
A key ingredient  in the proof of Theorem~\ref{thm-short} (and Theorem~\ref{thm-long1} below) is the  construction of a Sobolev-Dolbeault analog of the
relative cohomology of a topological space with respect to a subspace. We call this analog the
\emph{extendable cohomology} (see Section~\ref{sec-three-realizations} below). The analogous construction for the de Rham cohomology is well-known
(see \cite[Chapitre XII]{godbillon}). In many ways, the extendable cohomology, while easy to define, is pathological: for example, the $\dbar$-operator defining it is
not a closed operator, unlike the standard maximal or minimal realizations. What still allows us to prove results like the above is the the fact
that the extendable cohomology of Sobolev order $s$ is isomorphic to the usual Sobolev cohomology of order $s+1$ (Theorem~\ref{thm-istar} below). The key ingredient here is interior regularity of the $\dbar$-problem.

A second result of the same type as Theorem~\ref{thm-short}
 can be obtained by applying the same ideas to  a different short exact sequence \eqref{eq-sesq2}, but this time we obtain a single long exact sequence \eqref{eq-long-hole}  rather than a separate
short exact sequence for each $q$ . We state and prove this result in a general form in Section~\ref{sec-long1} below, using
the notion of a mixed realization, but note here the following consequence:
	\begin{cor}[to Theorem~\ref{thm-long1} below; see Section~\ref{sec-longcor}]	\label{cor-long-c}
		Let $\Omega=\Omega_1\setminus \ol{\Omega}_2\subset \cx^n$  be a bounded annulus. Let $s$ be an integer, and let $0\leq p\leq n$. The following sequence of semi-inner-product spaces
		and continuous {linear} maps is exact:
	\begin{equation}\label{eq-long-hole-min} \cdots \xrightarrow{{S}^{p,q-1}} H^{p,q-1}_{W^{s+1}}(\Omega_2)
	\xrightarrow{\ell^{p,q-1}}H^{p,q}_{c,W^s}(\Omega) \xrightarrow{(\eh_*)^{p,q}}H^{p,q}_{c,W^s}(\Omega_1) \xrightarrow{{S}^{p,q}}H^{p,q}_{W^{s+1}}(\Omega_2)
	\xrightarrow{\ell^{p,q}}\cdots\end{equation}
	where the maps $S^{p,q}, \ell^{p,q}$ and $(\eh_*)^{p,q}$ are defined in Section~\ref{sec-holerest} below.
\end{cor}

Mayer-Vietoris type arguments as  above are, of course, classical in complex analysis (see, e.g.,  \cite{andreo72a, andreo72b, andreo76}).

Section~\ref{sec-applications} is devoted to applications of the exact sequences established in Theorems~\ref{thm-short} and \ref{thm-long1} to questions about function theory
on annuli. As preliminaries, we discuss the Sobolev analog of Serre duality as well as recall the vanishing results for Sobolev cohomologies of pseudoconvex domains which will be used.
Instead of trying to state the most general results in this vein, we  consider examples illuminating the possibilities and limitations of our method. We prove Sobolev versions of several
results of \cite{shawclosed, meichiharmonic, lishaw, meichichristine, ChShTh, FLS} on the way.

While in Section~\ref{sec-applications}, we used vanishing results for cohomology on pseudoconvex domains in conjunction with our exact sequences to obtain information about annuli,
 in Section~\ref{sec-prescribed}, we change the point of view and apply known results about annuli, along with our exact sequences, to deduce new results about the cohomology of pseudoconvex domains. In particular, we obtain vanishing results for the minimal $W^s$-cohomologies of pseudoconvex domains and annuli (Propositions~\ref{prop-prescribed1}  and \ref{prop-prescribed2} below), i.e., we are able to solve the $\dbar$-problem in $\cx^n$, with estimates in a Sobolev space $W^s$,  with prescribed support in a pseudoconvex domain or in an annulus. By dualizing, we also
 are able to obtain, for the first time, solutions of the $\dbar$-problem with estimates in Sobolev spaces of negative order (see Corollaries~\ref{cor-negative1} and \ref{cor-negative2} below).
  It is well-known
 that  solvability with prescribed support is closely related to the extension of $\dbar_b$-closed forms on the boundary of a domain, and in  particular to the Hartogs-Bochner phenomenon for
 degree $(p,1)$ (see \cite[Section~9.1]{chenshaw,ShTh} and Proposition~\ref{prop-connected-complement} below). When we are in the $L^2$-setting, such a solution with $L^2$-estimates can be constructed starting from the $L^2$-canonical solution operator. However, we show that this $L^2$-solution in general does not admit Sobolev estimates for $s\leq -2$ (see Proposition~\ref{prop-notcts} below).

We have therefore  obtained in Proposition~\ref{prop-prescribed1}  a solution to the $\dbar$-problem which shows better behavior in
some Sobolev spaces  than the canonical solution. Perhaps one should not be surprised by this, since the canonical solution is the solution of the
$\dbar$-problem with smallest $L^2$-norm, so if it exists it is only guaranteed to be regular in the $L^2$-sense. The fact that it sometimes satisfies Sobolev estimates is a consequence of Sobolev estimates on the $\dbar$-Neumann operator, which may or may or may not hold. It is well-known that by using a strictly plurisubharmonic weight smooth up to the boundary, one can obtain Sobolev estimates
for the $\dbar$-problem for the weighted canonical solution (see \cite{kohn73}).  Even this technique does not work in situations like the polydisc (see \cite{ehsani1}). However, in \cite{ChShTh}, a $W^1$-Sobolev estimate was obtained for the $\dbar$-problem on the polydisc, using an argument in many ways similar to that in Proposition~\ref{prop-prescribed1}.
This raises the question of the relation of the methods of this paper with an old and difficult question, that of developing a $W^s$-theory of the  $\dbar$-problem, in analogy
to the classical $L^2$-theory. The hypothetical basic estimate of such a theory would bound $\norm{\dbar u}_{W^s}^2 + \norm{\dbar^*_{W^s}u}_{W^s}^2$ from below, where
$u$ is a form on a pseudoconvex domain $D$, which lies in $A^{p,q}_{W^s}(D)\cap \dom(\dbar^*_{W^s})$, where $A^{p,q}_{W^s}(D)$ is the domain of the maximal realization of
$\dbar$ as defined in Section~\ref{sec-realizations}	below,  and $\dbar^*_{W^s}$ is the adjoint of the maximal realization with respect in the $W^s$-space. Notice that on smoothly bounded
pseudoconvex domains, the regularity of the weighted canonical solution in Sobolev spaces already
implies that such a Sobolev basic estimate must hold. But, based on the Sobolev estimates on polydiscs and on solutions with prescribed support,  we can perhaps suspect that such an estimate holds
 on much wider classes of pseudoconvex domains.
While the question of $W^s$-estimates is very natural, at present very little is known regarding this problem in general domains (see, however,  \cite{boas1984, boas1985, krantz1,krantz2}).

\section{Sobolev Realizations of the \texorpdfstring{$\dbar$}{dbar}-operator}\label{sec-realizations}
\subsection{Sobolev spaces}

In this paper, we  consider the $\dbar$-operator acting  distributionally on certain Sobolev spaces of currents  on a domain $D\subset \cx^n$.
We call such an operator  a \emph{realization} of the $\dbar$-operator.
We use Sobolev spaces on domains which are not necessarily smooth or even Lipschitz, and also Sobolev spaces of negative index.
Extensive information on Sobolev spaces may be found in texts such as \cite{lionsmagenes} and, for Lipschitz domains, \cite{grisvard}. Here we
recall some definitions and facts, and set up notation.

 For real $s$, the Sobolev space $W^s(\rl^d)$ consists of tempered distributions whose Fourier transforms
 satisfy  the condition
 \[W^s(\rl^d)= \left\{ f\in \mathscr{S}'(\rl^d):  \wh{f}\in L^1_{\mathrm{loc}}(\rl^d) \text{ and }\norm{f}_{W^s(\rl^d)}^2:=\int_{\rl^d} \abs{ \wh{f}(\xi)}^2 (1+ \abs{\xi}^2)^s dV(\xi) <\infty \right\}. \]

% where we use a normalization of the Fourier transform in which for integrable $f$ we have
% $\wh{f}(\xi)= \int_{\rl^d}f(x)e^{-2\pi i \langle	x,\xi \rangle } dV(x).$
    It is well-known that
    \begin{itemize}
    	\item $W^s(\rl^d)$  is a Hilbert space.
    	\item  if $s$ is a non-negative integer  then  $W^s(\rl^d)$ can
    	also be thought of as the space of functions all of whose partial derivatives of order up to $s$ are in $L^2(\rl^d)$.
    	\item The dual of $W^s(\rl^d)$ can be conjugate-linearly and isometrically identified with $W^{-s}(\rl^d)$ by the pairing $f, g \mapsto \int_{\rl^d} \widehat{f} \ol{\wh{g}}dV$.
    	\end{itemize}
    Here, we will usually only consider the case when $s$ is an integer.  {This is to avoid certain pathologies that arise when $s-\frac{1}{2}$ is an integer.  In particular, we need to know that the closure of the space of compactly supported functions with respect to the Sobolev norm is equivalent to the space of functions which remain in the Sobolev space when extended by zero; see Corollary 1.4.4.5 in \cite{grisvard}.} We note, however, that many of our results (including the main exact sequences \eqref{eq-short-main} and \eqref{eq-long-hole}) do continue to hold for arbitrary real $s$.

 For a   domain $D\subset \rl^d$, we let $W^s(\overline{D})$ mean the space of restrictions of distributions in $W^s(\rl^d)$ to $D$, which is a Hilbert
 space with the norm
\[\norm{f}_{W^s(\overline{D})}= \inf_{\substack{F\in W^s(\rl^d)\\ F|_D=f}}\norm{F}_{W^s(\rl^d)}. \]
This is one of the standard definitions of Sobolev spaces on domains if $s\geq 0$, when it is standard to denote it by $W^s(D)$ (see, e.g., \cite{chenshaw}).  {We have adopted the notation of \cite{grisvard} (see Theorem 1.4.3.1 in \cite{grisvard} for proof that this definition is equivalent to other standard definitions when $s>0$ and $D$ is bounded with Lipschitz boundary).}  We emphasize here that we use the definition above also for negative $s$.  {When $s<0$, it is more conventional to define $W^s(D)$ to be the dual to $W^{-s}_0(D)$.  However, this space does not necessarily admit a bounded linear extension operator $E:W^{s}(D)\rightarrow W^s(\mathbb{R}^n)$, which will be critical to our constructions.  Hence, we work with the space $W^s(\overline{D})$, which admits an extension operator by definition.}

We can isometrically identify $W^s(\overline{D})$ with the quotient Hilbert space $W^s(\rl^d)/Z^s_D$, where $Z^s_D\subset W^s(\rl^d)$ is the closed
subspace given by
\[Z^s_D= \{f\in W^s(\rl^d): f|_D=0\}.\]
Thanks to the isomorphism $W^s(\overline{D})=(Z^s_D)^\perp$, which identifies $W^s(\overline{D})$ as a closed subspace of $W^s(\rl^d)$, we obtain a bounded linear extension operator
\begin{equation}
\label{eq-extension}
E: W^s(\overline{D})\to W^s(\rl^d),
\end{equation}
which associates with an $f\in  W^s(\overline{D})$ an element $F\in W^s(\rl^d)$ such that $F|_D=f$ and $\norm{F}_{W^s(\rl^d)}= \norm{f}_{W^s(\overline{D})}$.  {In particular, $Ef$ will be the unique extension of $f$ of minimal norm in $W^s(\rl^d)$.}

The isomorphism $W^s(\overline{D})\cong W^s(\rl^d)/Z^s_D$ gives the description of the dual
\begin{align*}
	W^s(\overline{D})'&= \{ \phi\in W^s(\rl^d)': \phi|_{Z^s_D}=0\}\\
	&=\{ \phi\in W^s(\rl^d)': \text{for $f\in W^s(\rl^d)$,  if } f|_D=0 \text{ then } \phi(f)=0 \}.
\end{align*}
Notice the smooth compactly supported functions are in $W^s(\rl^d)$.
Identifying $W^s(\rl^d)'$ with $W^{-s}(\rl^d)$ via the $L^2$-pairing, we have an isomorphism
\[	W^s(\overline{D})'\cong \{\phi\in W^{-s}(\rl^d): \mathrm{support}(\phi) \subset \overline{D}\}. \]
%Consequently, we can identify $W^s(\overline{D})'$ with the space
%\[ \{\phi \in W^{-s}(\ol{D}): \text{there is a } \Phi\in W^{-s}(\rl^d)  \text{ such that } \Phi|_D=\phi \text{ and } \Phi|_{\rl^d\setminus \ol{D}}=0   \}. \]

We let ${W}^t_0(\ol{D})$ be the subspace of $W^t(\overline{D})$ consisting of the images in $W^t(\overline{D})$  of  those $f\in W^t(\rl^d)$
  which are supported in $\overline{D}$. If we think of ${W}^t_0(\ol{D})$  as a space of ``distributions on $\overline{D}$," an element $f\in W^t_0(\overline{D})$   has a \emph{zero extension}, i.e.,
a distribution $\zs f\in W^t(\rl^d)$ which vanishes on $\rl^d \setminus \ol{D}$ and coincides with  $f$  on $\ol{D}$. Then we have the isomorphism
\[W^s(\overline{D})' = W^{-s}_0(\overline{D}),  \]
where the pairing of the two spaces is given by
\[ f, g \mapsto  \int_{\rl^d}\widehat{E f} \cdot \ol{\widehat{\zs g}}\,dV.\]

Recall other standard definitions of Sobolev spaces on domains.
For a positive integer $s$, let $W^s(D)$  be the space of functions on $D$  all whose partial derivatives of order up to $s$ are in $L^2(D)$. This is a Hilbert space  with the well-known standard norm.
 Let $W^s_0(D)$ be the closure in $W^{s}(D)$ of the  subspace $\mathcal{D}(D)$ of
 compactly supported forms. For $s<0$, let $W^s(D)$ be the dual of $W^{-s}_0(D)$.
 There is an injective continuous map $W^s(\ol{D})\to W^s(D)$ given by $f\mapsto Ef|_D$, where $E$ is as in \eqref{eq-extension}. Via this map, we will consider
 $W^s(\overline{D})$ as a subspace of $W^s(D)$. We note the following standard facts (see \cite{grisvard}):

 \begin{prop}
 	\label{prop-sobolevprops}

  The following hold:
\begin{enumerate}
	\item $W^s(\overline{D}) \subset W^s(D)\subset \mathcal{D}'(D)$, with continuous inclusions.
	\item $\mathcal{C}^\infty(\ol{D})$ is dense in $W^s(\overline{D})$ and $\mathcal{D}({D})$ is dense in $W^s_0(\overline{D})$
	\item if $s\geq 0$ and $D$ is Lipschitz then $W^s(\overline{D}) =W^s(D)$, with equivalent norms.
%	\begin{itemize}
%		\item 	
%	%	\item $\mathcal{C}^\infty(\ol{D})$ is dense in the space
%	\end{itemize}

	\item If  $s\leq\frac{1}{2}$  {and $D$ is Lipschitz} then   $W^s(\overline{D})={W}^s_0(\ol{D})$ so for $s\leq \frac{1}{2}$  the extension operator $E$ of \eqref{eq-extension} coincides with the zero extension operator.
\end{enumerate}
  \end{prop}

\subsection{The three standard realizations}\label{sec-three-realizations}
%Let $s$ be an integer, and  let $D$ be a \cancel{Lipschitz} domain in $\cx^n$.
%We denote by $W^s(D)$ the $L^2$-Sobolev space of order $s$ on $D$. There are
%many possible definitions of this space, but on Lipschitz domains they all coincide (see \cite{lionsmagenes}).

Let $D$ be a bounded domain in $\cx^n$. We let $W^s_{p,q}(\ol{D})$ and $(W^s_0)_{p,q}(\ol{D})$ denote the Hilbert spaces of currents with coefficients
in $W^s(\ol{D})$ and $W^s_0(\ol{D})$ respectively.  We introduce three standard realizations of the $\dbar$-operator on Sobolev spaces
on $D$ which will play a central role in what follows. Each realization is specified by defining its domain $A^{p,q}_{\db, W^s}(D)$, which is
a subspace of the space $W^s_{p,q}(\ol{D})$ of currents whose coefficients belong to the Sobolev space $W^s(\ol{D})$.

\begin{enumerate}[wide]%, labelwidth=!, labelindent=5pt]
	\item  For $0\leq p, q \leq n$, let $A^{p,q}_{W^s}(D)$ consist of those $u\in  W^s_{p,{q}}(\ol{D})$ such that $\dbar u\in  W^s_{p,{q+1}}(\ol{D})$, where $\dbar u$ is computed in the sense of distributions.
The \emph{$W^s$-graph norm}    of  $f\in A^{p,q}_{W^s}(D)$ is defined by
\begin{equation}
\label{eq-graphnorm}
\norm{f}_{A^{p,q}_{W^s}(D)}^2 =\norm{f}^2_{W^s(\ol{D})} + \norm{\dbar f}^2_{W^s(\ol{D})},
\end{equation}
which makes $A^{p,q}_{W^s}(D)$ into a normed linear space, and an inner product space with the obvious inner product that induces this norm.
We say that $A^{p,q}_{W^s}(D)$ is the \emph{domain of the (maximal)  $W^s$-realization} of the $\dbar$-operator, or simply the Sobolev realization if the order $s$ is understood.
%  We can think of the $W^s$-realization of the $\dbar$-operator also as a closed, densely defined operator on the Hilbert space $W^s_{p,q}(\ol{D})$

%{We denote by $A^{p,q}_{c, W^s}(D)$ the closure of {$\mathcal{D}^{p,q}(D)$, the space of} smooth compactly supported forms of degree $(p,q)$, in $A^{p,q}_{W^s}(D)$ with
%respect to the $W^s$-graph norm. This
%is the domain of the \emph{minimal Sobolev realization}, and the $\dbar$-operator acting on $A^{p,q}_{c, W^s}(D)$ distributionally is the
%\emph{minimal realization} of the $\dbar$-operator.}

\item  We denote

 \[ A^{p,q}_{c, W^s}(D)=  A^{p,q}_{W^s}(D)\cap (W^s_0)_{p,q}(\ol{D})
 .
 \]
In other words, $A^{p,q}_{c, W^s}(D)$ consists of those $f\in A^{p,q}_{W^s}(D)$
for which there is an $F\in A^{p,q}_{W^s}(\cx^n) $  such that $F|_D=f$, and $F\equiv 0$ on $\cx^n\setminus \ol{D}$.
%Notice that then $\dbar F \in (W^s_{0})_{p,q}(\ol{D})$, the space
%$A^{p,q}_{c, W^s}(D)$ is  a closed subspace of $A^{p,q}_{W^s}(D)$ and therefore
%a Hilbert space, and this defines $\dbar$ as  a closed, densely defined   operator on the Hilbert space $W^s_{p,q}(\ol{D})$.  We will refer to this as the
%\emph{minimal (Sobolev) realization} of $\dbar$.

Notice that by definition, there is  bounded linear \emph{zero extension operator}
\begin{equation}
\label{eq-zeroext} \zs: A^{p,q}_{c, W^s}(D)\to A^{p,q}_{W^s}(\cx^n),
\end{equation}
which maps a current $f$ to a current $\zs f$ such that $\zs f|_D=f$ and $\zs f|_{\cx \setminus \ol{D}}=0$.

\item We define the \emph{extendable (Sobolev) realization}
$A^{p,q}_{\mathrm{ext}, W^s}(D)$ to consist of those currents $f\in A^{p,q}_{W^s}(D)$ for which there exists an $F\in A^{p,q}_{ W^s}(\cx^n)$
such that $F|_D=f$.  We endow $A^{p,q}_{\mathrm{ext}, W^s}(D)$ with the graph norm \eqref{eq-graphnorm}, which makes it into an inner-product space which is not necessarily complete. Notice that we have the inclusions of subspaces
\[A^{p,q}_{c, W^s}(D) \subset A^{p,q}_{\mathrm{ext}, W^s}(D)\subset A^{p,q}_{W^s}(D),\]
where (depending on $s$ and $D$) the inclusions can be strict. Notice that $ A^{p,q}_{\mathrm{ext}, W^s}(D)\neq A^{p,q}_{W^s}(D)$ means precisely
that there is a current $f\in A^{p,q}_{W^s}(D)$ such that $f$ extends to an element $F$ of $A^{p,q}_{W^s}(\cx^n)$ and
$\dbar f$ extends to an element $G$ of $A^{p,q+1}_{W^s}(D)$, but we cannot have $\dbar F=G$ for such extensions.

\end{enumerate}

\subsection{ The associated cochain complex}
\label{sec-cochains}
Let $0\leq p \leq n$, and let   $A^{p,q}_{\db,W^s}(D)$  stand for any one of $A^{p,q}_{W^s}(D)$, $A^{p,q}_{c,W^s}(D)$ or $A^{p,q}_{\mathrm{ext},W^s}(D)$ (the same one for each $q$).
Each of the  three $W^s$-realizations of the $\dbar$-operator on a domain $D\subset \cx^n$  defined in Section~\ref{sec-three-realizations} defines   a cochain complex in the sense of homological algebra (see \cite[Chapter XX, $\S$1 ]{lang}) :
\begin{equation}
\label{eq-realization}   A^{p,0}_{\db,W^s}(D)\xrightarrow{\dbar} A^{p,1}_{\db,W^s}(D)\xrightarrow{\dbar}\cdots \xrightarrow{\dbar} A^{p,n}_{\db,W^s}(D),
\end{equation}
where in each case $\dbar$ acts in the sense of distributions.
Endowed with the inner product
corresponding to the graph norm \eqref{eq-graphnorm}, this is a cochain complex of inner-product spaces.  Notice that the differential $\dbar$ is continuous in each degree from the definition.
We denote this cochain complex by $A^{p,*}_{\db, W^s}(D)$.

The space of \emph{cocycles} defined by
\[ Z^{p,q}_{\db, W^s}(D)= \ker\left\{\dbar: A^{p,q}_{\db,W^s}(D) \to A^{p,q+1}_{\db,W^s}(D)\right\}\]
is a closed subspace of $A^{p,q}_{\db,W^s}(D)$, by the continuity of $\dbar$ in the graph norm.  We say that a realization is \emph{closed} if each $A^{p,q}_{\db,W^s}(D)$ is a Hilbert space in the graph norm. It is not difficult to  see that $A^{p,*}_{W^s}(D)$ and $A^{p,*}_{c,W^s}(D)$ are closed (see below, Proposition~\ref{prop-maxrealproperties}, part~\ref{part-2}). On the other hand one can show that the extendable realization $A^{p,*}_{\mathrm{ext}, W^s}(D)$ is not closed in general. For a  closed realization, $Z^{p,q}_{\db, W^s}(D)$ is also a
closed subspace of $W^s_{p,q}(\ol{D})$ (and therefore a Hilbert space in the norm of $W^s_{p,q}(\ol{D})$).

 The space of \emph{coboundaries}
\[ B^{p,q}_{\db, W^s}(D)= \img\left\{\dbar: A^{p,q-1}_{\db,W^s}(D) \to A^{p,q}_{\db,W^s}(D)\right\}\]
is contained in  $Z^{p,q}_{\db, W^s}(D)$ since $\dbar^2=0$. The \emph{cohomology groups} of the  complex \eqref{eq-realization}  are the vector spaces
\[ H^{p,q}_{\db, W^s}(D)=  Z^{p,q}_{\db, W^s}(D)/B^{p,q}_{\db, W^s}(D). \]
As the quotient of the inner product space $ Z^{p,q}_{\db, W^s}(D)$ by the subspace $B^{p,q}_{\db, W^s}(D)$,
the cohomology group $ H^{p,q}_{\db, W^s}(D)$ has the structure of a \emph{semi-inner-product (SIP) space}, as
explained in Proposition~\ref{prop-quotient} below, where it is shown that $H^{p,q}_{\db, W^s}(D)$ has  a natural sesquilinear form $\ipr{,}$ (the semi-inner-product), which differs from a genuine inner-product only in the fact that $\ipr{x,x}^{\frac{1}{2}}$ is a semi-norm, and not necessarily a norm (i.e., we may have $\ipr{x,x}^{\frac{1}{2}}=0$ for $x\neq 0$).

Since $H^{p,q}_{\db, W^s}(D)$ is the quotient of two topological vector spaces, it has the natural quotient
topology. It is well-known that this topology is not necessarily Hausdorff. In fact, it is Hausdorff if and only if
$ B^{p,q}_{\db, W^s}(D)$ is closed as a subspace of $Z^{p,q}_{\db, W^s}(D)$. It will be seen from Proposition~\ref{prop-quotient}
below that the quotient topology of  $H^{p,q}_{\db, W^s}(D)$  is
also induced by the semi-inner-product through its associated semi-norm. Therefore, the use of the semi-inner-product structure provides a
concrete approach to working with the otherwise pathological non-Hausdorff topologies that one encounters in this investigation.

We denote by
\[ H^{p,q}_{W^s}(D), H^{p,q}_{c,W^s}(D), H^{p,q}_{\mathrm{ext},W^s}(D)\]
 the  \emph{$W^s$-cohomology,  the minimal $W^s$-cohomology, and the extendable cohomology} respectively, which are by definition
 the cohomologies associated with the three realizations introduced in  Section~\ref{sec-three-realizations}.
\subsection{Some basic properties}
\begin{prop}\label{prop-maxrealproperties}
	Let $D$ be a bounded domain in $\cx^n$, let $0\leq p \leq n$ and let $s$ be an integer.
	\begin{enumerate}
		\item \label{part-1}	We have
		\[ A^{p,0}_{W^s}(D) \subset (W^{s+1}_{\mathrm{loc}})_{p,0}(D) \cap W^s_{p,0}(\ol{D}).\]
		\item  \label{part-2} For each {$0\leq q\leq n$}, the inner product spaces $A^{p,q}_{W^s}(D)$ and $A^{p,q}_{c,W^s}(D)$ are Hilbert spaces, i.e., each of them
		is complete in the graph norm; consequently, both the Sobolev  and minimal realizations are closed.
		
		\item \label{part-3} Assume that $D$ is Lipschitz.  In the graph norm  \eqref{eq-graphnorm},  the subspace $\mathcal{C}^\infty_{p,q}(\ol{D})$ of forms smooth up to the boundary
			is dense in $A^{p,q}_{W^s}(D)$.
		
		\item\label{part-4} Assume that $D$ is Lipschitz. In the graph norm  \eqref{eq-graphnorm},  the subspace $\mathcal{D}^{p,q}({D})$ of smooth compactly supported forms
			is dense in $A^{p,q}_{c,W^s}(D)$.
	
		\item \label{part-6} $H^{p,0}_{c,W^s}(D)=0$.
		\item \label{part-7}	Let $D'$ be a bounded open set such that  $D\Subset D'\subset \cx^n$. Then for each $f\in A^{p,q}_{\mathrm{ext},W^s}(D)$, there is an $\wt{f}\in A^{p,q}_{c, W^s}(D')$, the minimal realization on $D'$, such  that $\wt{f}|_D=f$.
			\item \label{part-5 } Let $D\subset \cx^n$ be a bounded domain, let $0\leq p \leq n$, and let $s$ be an integer. Then:
		\[ W^{s+1}_{p,0}(\ol{D})= A^{p,0}_{\mathrm{ext}, W^s}(D). \]
	\end{enumerate}
\end{prop}

\begin{proof}
	\begin{enumerate}[wide, labelwidth=!, labelindent=0pt]
		\item

	Notice that by definition  a form belongs to $A^{p,0}_{W^s}(D)$ if and only if each coefficient function belongs to the space of functions $A^{0,0}_{W^s}(D)$. Therefore,
	we can assume without loss of generality that $p=0$.
	
	Let  $U$ and $V$ be open subsets  such that
	$ U\Subset V \Subset D,$ and let $\chi$ be a smooth compactly supported function on $\cx^n$ such that $\chi\equiv 1$ on $U$ and $\chi$ is supported inside $V$.  Now let $f\in A^{0,0}_{W^s}(D)$, and
	set $g=\chi\cdot f$, where $g$ is understood to be extended by zero outside $D$. Then since $\wh{g}= \wh{\chi}* \wh{f}$,  we easily conclude
	that ${g}\in W^s(\cx^n)$. Similarly, we can see that $\dbar g \in W^s_{0,1}(\cx^n)$, i.e., for each $j$, the derivative
	\begin{equation}\label{eq-dgdz}
	\dfrac{\partial g}{\partial \ol{z_j}}\in W^s(\cx^n).
	\end{equation}
	Denoting as usual the coordinates of $\cx^n$ by $(z_1,\dots, z_n)$, with $z_j=x_j+iy_j$, we denote the corresponding Fourier variables
	by $\zeta_j=\xi_j+i \eta_j$. Then notice that
	\[\wh{\dfrac{\partial g}{\partial \ol{z_j}}} = \frac{1}{2}\left\{\left(\frac{\partial}{\partial x_j}+ i \frac{\partial}{\partial y_j}\right)g\right\} \widehat{\phantom{\frac{\partial}{\partial}}}= \frac{1}{2}\cdot 2\pi i (\xi_j+ i \eta_j)\wh{g}=\pi i \zeta_j \wh{g}. \]
	Therefore \eqref{eq-dgdz} gives that for each $j$ we have
	\[ \int_{\cx^n} \abs{\zeta_j}^2 \abs{\wh{g}(\zeta)}^2 (1+\abs{\zeta}^2)^s dV(\zeta) <\infty. \]
	Summing this from $j=1$ to $n$, and also adding the inequality  $ \int_{\cx^n} \abs{\wh{g}(\zeta)}^2 (1+\abs{\zeta}^2)^s dV(\zeta) <\infty $
	(since $g\in W^s(\cx^n)$), we conclude that
	\[  \int_{\cx^n} \abs{\wh{g}(\zeta)}^2 (1+\abs{\zeta}^2)^{s+1} dV(\zeta) <\infty,\]
	i.e., $g\in W^{s+1}(\cx^n)$.
	
	Notice that $g|_U = f|_U$. Therefore,
	each point  of $D$ has a neighborhood $U$ such that $f|_U \in W^{s+1}(U)$. It follows
	that $f\in W^{s+1}_{\rm loc}(D)$.
	
	\item   Let $\{f_j\}\subset A^{p,q}_{W^s}(D)$ be a  Cauchy sequence in the graph norm. Then there exist
	$f\in W^s_{p,q}(\ol{D})$ and $g\in W^{s+1}_{p,q+1}(\ol{D})$ such that
	$f_j \to f$ in $W^s_{p,q}(\ol{D})$ and $\dbar f_j \to g$ in
	$W^s_{p,q+1}(\ol{D})$. Thanks to the continuous inclusion of the Sobolev space $W^s_{p,q}(\ol{D})$ in the space $\mathcal{D}'_{p,q}(D)$
	of currents, the latter assumption implies  $g=\dbar f$, so $f_j\to f$ in the graph norm.
	
	Since $A^{p,q}_{c,W^s}(D)$
	is a closed subspace of $A^{p,q}_{W^s}(D)$, it is  therefore a Hilbert space in the subspace topology.
	
	\item The proof follows the same lines as the classical argument
		for $s=0$, which may be found, e.g., in \cite[Proposition~2.3, part (ii)]{straubebook}.  After using a partition of unity on $\ol{D}$, we only
		need to consider the case of a form $u$ supported compactly
		in a neighborhood   $U$  in $\ol{\Omega}$ of a point $P\in b\Omega$.  Since the boundary $b\Omega$ is
		Lipschitz, there is a cone $\Gamma$ in $\cx^n$ with vertex at the origin and an $a>0$ such that whenever {$z\in U\cap\Omega$}, $\zeta \in \Gamma$, and $\abs{\zeta}<a$ we
		have $z-\zeta \in \Omega$. We choose a cutoff $\phi\in \mathcal{C}^\infty_0(\Gamma\cap B(0,a))$ with $\phi\geq 0$ and $\int\phi =1$, where
		$B(0,a)$ is an open ball with center at the origin and radius $a>0$. For $\epsilon>0$ set $\phi_\epsilon (z)= \epsilon^{-2n}\phi\left(\frac{z}{\epsilon}\right)$,
		and let {$E u\in W^s_{p,q}(\cx^n)$ be the extension given by \eqref{eq-extension}} (recall here that $u$ and $\dbar u$ have coefficients in $W^s(\ol{U})$). Then one can verify that as $\epsilon\to 0$, the restriction
		$(\phi_\epsilon\ast E u)|_{\ol{\Omega}} \to u$ in the graph norm. For details see \cite{straubebook}.
	
\item  {Here, we follow the proof of \cite[Lemma~4.3.2, part (ii)]{chenshaw}.  Using a partition of unity, we may assume that $\ol{D}$ is star-shaped with a center at $0$.  For $u\in A^{p,q}_{c,W^s}(D)$, we let $\zs u$ denote the zero extension as in \eqref{eq-zeroext}, and note that $\dbar\zs u=\zs\dbar u$ by definition.  For $\epsilon>0$, we define $\tilde u^{-\epsilon}(z)=\zs u\left(\frac{z}{1-\epsilon}\right)$, so that $\tilde u^{-\epsilon}$ is compactly supported in $D$ and $\tilde u^{-\epsilon}\rightarrow\zs u$ in the graph norm.  Regularizing this approximating sequence by convolution will complete the proof.}

	\item If $f\in Z^{p,0}_{c,W^s}(D)= H^{p,0}_{c,W^s}(D)$,
	then its zero-extension is a compactly supported holomorphic {current} on $\cx^n$, and therefore a compactly supported
	holomorphic form by Weyl's lemma. This vanishes by the identity principle.
	
	\item By definition, there is an ${f_0}\in A^{p,q}_{W^s}({\cx^n})$ such that $f = {f_0}|_D$.  Let $\chi\in \mathcal{C}^\infty_0(\cx^n)$ be a cutoff such that $\chi\equiv 1$ on $D$ and $\chi\equiv 0$ off $D'$.  Then we can take $\wt{f}= \chi\cdot {f_0}$.
	
		\item 	Let  $f\in A^{p,0}_{\mathrm{ext},W^s}(D)$, so that by Part~\ref{part-7} above  there is an
	$\wt{f}\in A^{p,0}_{c,W^s}(D')$ such that $f =\wt{f}|_D$.  By Part~\ref{part-1} above, $\wt{f}\in (W^{s+1}_{\rm loc})_{p,0}(D')$. Therefore, $f=\wt{f}|_D\in W^{s+1}_{p,0}(\ol{D})$.
	
		\end{enumerate}
\end{proof}

\subsection{Relation of extendable cohomology with Sobolev cohomology}

The following result is at the heart of our approach to function theory on annuli:

\begin{thm} \label{thm-istar} Let $D\subset\cx^n$ be a bounded  domain,  let $0\leq p \leq n$, and let $s\in \mathbb{Z}$.
	Then the inclusion map
	\begin{equation}\label{eq-i}
	i:   A^{p,*}_{W^{s+1}}(D) \hookrightarrow A^{p,*}_{\mathrm{ext}, W^s}(D)
	\end{equation}
	is a continuous injective cochain morphism.
	The induced linear map at the level of cohomology (see Proposition~\ref{prop-ind-cont} below)
	\begin{equation}\label{eq-istar}
	i_*^{p,q}: H^{p,q}_{W^{s+1}}(D)\to H^{p,q}_{\mathrm{ext},W^s}(D)
	\end{equation}
	is a continuous bijection of semi-inner-product spaces for each $q$. Let $U\supset \ol{D}$ be a  bounded pseudoconvex domain,  and $N$ be the $\dbar$-Neumann operator of  $U$. Then the inverse of the map  $i_*^{p,q}$ in  \eqref{eq-istar} is given  by (with $f\in Z^{p,q}_{\mathrm{ext}, W^s}(D)$)
	\begin{equation}
	\label{eq-istarinv}
	( i_*^{p,q})^{-1} \left(\cls{f}{H^{p,q}_{\mathrm{ext}, W^s}(D)}\right)=  \cls{(\dbar^* N\dbar \wt{f})|_{D}}{H^{p,q}_{W^{s+1}}(D)},
	\end{equation}
	where  $\wt{f}$ is an extension of $f$ as an element of $A^{p,q}_{c,W^s}(U)$ (see Part (7) of Proposition~\ref{prop-maxrealproperties}).
\end{thm}
\begin{rem}
	
	\begin{enumerate}[wide, labelwidth=!, labelindent=0pt]
		\item The map $( i_*^{p,q})^{-1}$ exists algebraically as a linear map but is  not known to be  continuous. This is because the extension operation $f\mapsto \wt{f}$
		is not known to be continuous.  {In fact, it is possible to construct explicit examples to show that $(i_*^{p,0})^{-1}$ is not continuous for any $0\leq p\leq n$.}
		\item Notice that the dependence of the right hand side of \eqref{eq-istarinv} on the pseudoconvex neighborhood $U$, its $\dbar$-Neumann operator $N$,
		and the particular extension $\wt{f}$ of $f$,		is illusory, since the left hand side is defined independently of $U$.
		\item If the topology on  $H^{p,q}_{\mathrm{ext},W^s}(D)$ is  Hausdorff,
		then it follows from the continuity of  $i_*^{p,q}$ that the topology on $H^{p,q}_{W^{s+1}}(D)$ is also Hausdorff. If these spaces are also complete, it  follows by the
		open-mapping/closed-graph theorem that $i_*^{p,q}$ is a linear homeomorphism.
        \item {In the following, we will use the $\dbar$-Neumann operator on currents with coefficients in $W^s(U)$ even when $s<0$.  For such currents, we may use the self-adjointness of the $\dbar$-Neumann operator to define its action distributionally, i.e., given $f\in W^s_{p,q}(U)$ for $s<0$, we define $(N f,\varphi)=(f,N\varphi)$ for all compactly supported smooth forms
        	 $\varphi\in \mathcal{D}^{p,q}(U)$, where $(\cdot,\cdot)$ is the natural extension of the $L^2$-inner product on $(p,q)$-forms by density to an action of a current of degree $(p,q)$ on a $(p,q)$-form of compact support.  {To obtain interior regularity for the canonical solution,} we choose an intermediate set $U'$ satisfying $D\Subset U'\Subset U$ and let $\chi\in C^\infty_0(U)$ be supported in $U$ and equal one identically on $U'$.  Then for $s\leq 0$ and $f\in W^s_{p,q}(U)$ supported in $U'$}
            \begin{align*}
              \norm{\dbar^* N f}_{W^{s}(D)}&=\sup_{\substack{\varphi\in \mathcal{D}^{p,q}(D)\\\varphi\neq 0}}\frac{(\dbar^*Nf,\varphi)_{L^2(D)}}{\|\varphi\|_{W^{-s}(D)}}
              =\sup_{\substack{\varphi\in \mathcal{D}^{p,q}(D)\\\varphi\neq 0}}\frac{(\dbar^*N(\chi f),\tilde\varphi)_{L^2(U)}}{\|\tilde\varphi\|_{W^{-s}(U)}}\\
              &=\sup_{\substack{\varphi\in \mathcal{D}^{p,q}(D)\\\varphi\neq 0}}\frac{(f,\chi N\dbar\tilde\varphi)_{L^2(U)}}{\|\tilde\varphi\|_{W^{-s}(U)}}
              \leq \sup_{\substack{\varphi\in \mathcal{D}^{p,q}(D)\\\varphi\neq 0}}\frac{\|f\|_{W^{s-1}(U)}\cdot \|\chi N\dbar\tilde\varphi\|_{W^{-s+1}(U)}}{\|\tilde\varphi\|_{W^{-s}(U)}}\\&\leq C\|f\|_{W^{s-1}(U)},
            \end{align*}
           where $\wt{\varphi}$ denotes the zero-extension of $\varphi$. Here, we have used the known interior regularity for the canonical solution operator to the $\dbar^*$ equation, $N\dbar$, in Sobolev spaces with non-negative index.
	\end{enumerate}
\end{rem}
\begin{proof}[Proof of Theorem~\ref{thm-istar}] Notice first that $W^{s+1}_{p,q}(\ol{D}) \subset  A^{p,q}_{\mathrm{ext}, W^s}({D}) $, since there is an extension operator $E:W^{s+1}_{p,q}(\ol{D}) \to W^{s+1}_{p,q}(\cx^n)$.  It follows that  $A^{p,*}_{W^{s+1}}(D) \subset A^{p,*}_{\mathrm{ext}, W^s}(D)$, and that
	 the inclusion \eqref{eq-i} is algebraically a cochain morphism, since in
	both cochain complexes, the differential is the $\dbar$-operator in the sense of distributions. Since the topology on the subspace
	$A^{p,*}_{W^{s+1}}(D)$ of $W^{s+1}_{p,*}(\ol{D})$  is the graph topology coming from $W^{s+1}_{p,*}(D)$ and
	that on $A^{p,*}_{\mathrm{ext}, W^s}(D)$ is the  graph topology coming from $W^s_{p,*}(D)$, it is clear that $i$ is continuous.

	The map $i_*$ is continuous since it is induced by the (continuous)  cochain morphism $i$ (see  Proposition~\ref{prop-ind-cont} ).
	To see that $i_*$ is a bijection, first, consider the case $q=0$. Then $H^{p,0}_{W^{s+1}}(D)=Z^{p,0}_{W^{s+1}}(D)$, i.e.,
	$H^{p,0}_{W^{s+1}}(D)$ is the space of holomorphic $p$-forms with coefficients in $W^{s+1}(\ol{D})$. Similarly $H^{p,0}_{\mathrm{ext}, W^s}(D)= Z^{p,0}_{\mathrm{ext}, W^s}(D)$, the space
	of holomorphic $p$-forms which also lie in $A^{p,0}_{\mathrm{ext}, W^s}(D)$.  Therefore, $i_*^{p,0}$ is simply the inclusion map
	
	\[ Z^{p,0}_{W^{s+1}}(D)\hookrightarrow Z^{p,0}_{\mathrm{ext},W^s}(D)\]
	which, however, is actually a bijection,  since $ Z^{p,0}_{W^{s+1}}(D)= A^{p,0}_{\mathrm{ext}, W^s}(D)\cap \ker \dbar$,
	and by part~\ref{part-5 } of Proposition~\ref{prop-maxrealproperties}, we have
	that  $W^{s+1}_{p,0}(\ol{D})=A^{p,0}_{\mathrm{ext}, W^s}(D)$

	Now consider the case $q\geq 1$. Let $f\in Z^{p,q}_{W^{s+1}}(D)$ be such that $\cls{f}{H^{p,q}_{W^{s+1}}(D)}\in \ker (i_*^{p,q})$, i.e., there is  a $u\in A^{p,q}_{\mathrm{ext},W^s}(D)$ such that $\dbar u =f$ on $D$.
	As in the  statement of the theorem, let $U$ be a bounded pseudoconvex domain containing $\ol{D}$,  let $\wt{u}\in A^{p,q-1}_{c,W^s}(U)$  have compact support in $U$ and $\wt{u}|_D=u$ (see  Part (7) of Proposition~\ref{prop-maxrealproperties}), and set $\wt{f}=\dbar \wt{u}$.
	Let $v=\dbar^* N \wt{f}$ be the canonical solution of the $\dbar$-problem $\dbar v=\wt{f}$ in $U$. (Notice that $\dbar \wt{f}=0$ in $U$, and as in the statement of the theorem, $N$ denotes the
	$\dbar$-Neumann operator of  $U$.) Then, by interior elliptic gain in the $\dbar$-Neumann problem,
	we have that  $v_0:=v|_D \in W^{s+1}_{p,q-1}(\ol{D})$, and $\dbar v_0=f\in W^{s+1}_{p,q}(\ol{D})$. Therefore
	$v_0\in A^{p,q-1}_{W^{s+1}}(D)$, so that $f\in B^{p,q}_{W^{s+1}}(D)$. It  now follows that the class
	$\cls{f}{H^{p,q}_{W^{s+1}}(D)}=0$, so that $\ker (i_*^{p,q})=0$, and $i_*^{p,q}$ is injective.

	To show that $i_*^{p,q}$ is surjective, we construct a right inverse.
	If $[f]_{H^{p,q}_{\mathrm{ext}, W^s}(D)}$ is a class in $H^{p,q}_{\mathrm{ext}, W^s}(D)$ where $f\in Z^{p,q}_{\mathrm{ext},W^s}(D)$, by
	Part~(7) of Proposition~\ref{prop-maxrealproperties} there is an $\wt{f}\in A^{p,q}_{c,W^s}(U)$ such that
	$\wt{f}|_D=f$. Let $g=\dbar \wt{f}$ so that (since $\wt{f} \in A^{p,q}_{c,W^s}(U)$) we have $g \in B^{p,q+1}_{c,W^s}(U)$,  and set
	$u=\dbar^* N g$, the canonical solution of $\dbar u =g$. Let
	\begin{equation}
	\label{eq-u0}
	u_0 =u|_D = (\dbar^* N g)|_{D}
	\end{equation}
	Then by interior regularity, $u_0 \in W^{s+1}_{p,q}(\ol{D})$ and $\dbar u_0=0$, so $u_0\in Z^{p,q}_{W^{s+1}}(D)$. Also,
	since $\dbar(u-\wt{f})=0$ on $U$, it follows that $u-\wt{f}=\dbar v$ on
	$U$, where  we can take $v= \dbar^* N (u-\wt{f})$. Then we have $v_0=v|_D\in A^{p,q-1}_{\mathrm{ext},W^s}(D)$, so that on $D$, we have $u_0=f+\dbar v_0$. Then we have
	\begin{align*}
	i_*^{p,q}\left( [u_0]_{H^{p,q}_{W^{s+1}}(D)}\right)&= [ i(u_0)]_{H^{p,q}_{\mathrm{ext},W^s}(D)}=[ u_0]_{H^{p,q}_{\mathrm{ext}, W^s}(D)}\\&
	=[f+\dbar v_0]_{H^{p,q}_{\mathrm{ext}, W^s}(D)}=[f]_{H^{p,q}_{\mathrm{ext},W^s}(D)}.
	\end{align*}
	It now follows that $i_*^{p,q}$ is surjective, and from \eqref{eq-u0}, noting that $g=\dbar \wt{f}$, \eqref{eq-istarinv} follows.
\end{proof}

\section{The exact sequences}

\subsection{Proof of Theorem~\ref{thm-short}}

Define the restriction operator $\rhoa$ on the space of currents on $\Omega_1$ by setting for $f$ a current on
$\Omega_1$:
\[ \rhoa(f)= f|_\Omega.\]
It is clear that if $f\in  A^{p,*}_{\mathrm{ext},W^s}(\Omega_1)$ then ${\rhoa(f)} \in A^{p,*}_{\mathrm{ext}, W^s}(\Omega)$.
Further, from the definition of $A^{p,*}_{\mathrm{ext}, W^s}(\Omega)$ it follows that the restriction map $\rhoa$ is surjective and continuous as
a map from  $A^{p,*}_{\mathrm{ext},W^s}(\Omega_1)$ to $A^{p,*}_{\mathrm{ext}, W^s}(\Omega)$.

Let $\ea$ denote the operator which extends currents on the hole $\Omega_2$  to the whole envelope $\Omega_1$ by setting them equal to zero on the annulus $\Omega$:
\[ \ea(f) = \begin{cases} f &\text{on } \Omega_2\\ 0 & \text{on } \Omega,
\end{cases}\]
provided that this defines a current on $\Omega_1$. In particular, it is clear that if $f\in A^{p,*}_{c,W^s}(\Omega_2)$ then
$ \ea(f)\in A^{p,*}_{c,W^s}(\Omega_1)\subset A^{p,*}_{\mathrm{ext},W^s}(\Omega_1)$. It is clear that the map $\ea$ is continuous and injective.

Notice  that
the sequence of cochain complexes of inner-product spaces and continuous cochain maps
\begin{equation}\label{eq-sesq1}
0\to A^{p,*}_{c,W^s}(\Omega_2)\xrightarrow{\ea} A^{p,*}_{\mathrm{ext},W^s}(\Omega_1)\xrightarrow{\rhoa}A^{p,*}_{\mathrm{ext}, W^s}(\Omega)\to 0.
\end{equation}
is exact.  In view of the comments in the previous paragraph, we only need to verify exactness at the middle term, i.e., $\img \ea= \ker \rhoa$. But both these subspaces of
$ A^{p,*}_{\mathrm{ext},W^s}(\Omega_1)$ consist of restriction to $\Omega_1$ of those $F\in A^{p,*}_{W^s}(\cx^n)$ whose support is in $\ol{\Omega_2}$.

Therefore  (see \cite[Chapter XX, Theorem 2.1]{lang})  we obtain a long exact sequence of semi-inner-product spaces and
linear maps:
\begin{equation}\label{eq-long2}
\cdots \xrightarrow{(\rhoa_*)^{p,q-1}} H^{p,q-1}_{\mathrm{ext}, W^s}(\Omega)
\xrightarrow{c_\mathrm{A}^{p,q-1}}H^{p,q}_{c,W^s}(\Omega_2) \xrightarrow{(\ea_*)^{p,q}}H^{p,q}_{\mathrm{ext},W^s}(\Omega_1) \xrightarrow{(\rhoa_*)^{p,q}}H^{p,q}_{\mathrm{ext},W^s}(\Omega)
\xrightarrow{{c_\mathrm{A}^{p,q}}}\cdots\end{equation}
where the continuous maps $\rhoa_*$ and $\ea_*$ are induced by the cochain maps $\rhoa$ and $\ea$ respectively, and
$c_\mathrm{A}$
is the connecting homomorphism. One can easily check using the definition that for $f\in Z^{p,q}_{\mathrm{ext}, W^s}(\Omega)$,
we have
\begin{equation}
\label{eq-c2def}
c_\mathrm{A}^{p,q} \left( [f]_{H^{p,q}_{\mathrm{ext},W^s}(\Omega)}\right)= \left[\left.\left(\dbar \wt{f}\right)\right\vert_{\Omega_2}\right]_{H^{p,q+1}_{c,W^s}(\Omega_2)},
\end{equation}
where $\wt{f}\in A^{p,q}_{\mathrm{ext}, W^s}(\Omega_1)$  is an extension of $f$ (see Part~7 of Proposition~\ref{prop-maxrealproperties}).

For the envelope $\Omega_1$ and the hole $\Omega$  let
\[ (i_*^{\Omega_1})^{p,q} : H^{p,q}_{W^{s+1}}(\Omega_1)\to H^{p,q}_{\mathrm{ext},W^s}(\Omega_1)\]
and
\[ (i_*^{\Omega})^{p,q} : H^{p,q}_{W^{s+1}}(\Omega)\to H^{p,q}_{\mathrm{ext},W^s}(\Omega)\]
be the continuous isomorphisms given by Theorem~\ref{thm-istar}. {If we let $R^{p,q}_*$ and $\lambda^{p,q}$ be as in the statement of Theorem \ref{thm-short}, then we} have the following:

\begin{lem}\label{lem-modi2}
	The following equalities hold, where we suppress the superscripts $p,q$ from all maps for simplicity:
	\begin{enumerate}
		\item \label{part-1modi2} $\rhoa_* \circ i_*^{\Omega_1}= i_*^{\Omega}\circ R_*$,
		\item \label{part-2modi2}$\lambda= c_\mathrm{A}\circ i_*^{\Omega}$,
		\item \label{part-3modi3} {$\ea_*=0$.}
	\end{enumerate}
\end{lem}
\begin{proof}
	\begin{enumerate}[wide, labelwidth=!, labelindent=0pt]
		\item Let $f\in Z^{p,q}_{W^{s+1}}(\Omega_1)$. Then a direct computation shows that
		\[ \rhoa_* \circ i_*^{\Omega_1}\left(\left[f\right]_{H^{p,q}_{W^{s+1}}(\Omega_1)}\right)=\left[f|_\Omega\right]_{H^{p,q}_{\mathrm{ext},W^s}(\Omega)}= i_*^{\Omega}\circ R_* \left(\left[f\right]_{H^{p,q}_{W^{s+1}}(\Omega_1)}\right)\]
		and thus the equality in Part~(\ref{part-1modi2}) holds.
		
		For Part~(\ref{part-2modi2}), let $f\in Z^{p,q}_{W^{s+1}}(\Omega)$. Then a direct computation shows that
		\begin{align*}
		(c_\mathrm{A}\circ i_*^{\Omega})([f]_{H^{p,q}_{W^{s+1}}(\Omega)})&=c_\mathrm{A} \left([f]_{H^{p,q}_{\mathrm{ext}, W^s}({\Omega})}\right) \\
		&= \left[\left.\left(\dbar \wt{f}\right)\right\vert_{\Omega_2}\right]_{H^{p,q+1}_{c,W^s}(\Omega_2)} &\text{ (where $\wt{f}\in A^{p,q}_{\mathrm{ext}, W^s}(\Omega_1)$  is an extension of $f$) }\\
		&= \left[\left.\left(\dbar E f\right)\right\vert_{\Omega_2}\right]_{H^{p,{q+1}}_{c, W^s}(\Omega_2)} & \text{with $E$ as in \eqref{eq-extension}}\\
		&=	\lambda\left( [f ]_{H^{p,q}_{W^{s+1}}(\Omega)}\right)&  \text{ see \eqref{eq-lambdadef}}.
		\end{align*}
		
		{For Part~(\ref{part-3modi3}), we first assume that $1\leq q\leq n$ and let $f\in Z^{p.q}_{c,W^s}(\Omega_2)$. Let $U$ be a bounded pseudoconvex domain containing $\overline{\Omega}_1$, let $\wt{f}\in Z^{p,q}_{c,W^s}(U)$ be the extension by zero of $f$ to $U$, and let  $N$ be the $\dbar$-Neumann operator on $U$.  Then
			\begin{align*}
			\ea_*[f]_{H^{p,q}_{c,W^s }(\Omega_2)}&=[\ea f]_{H^{p,q}_{\mathrm{ext},W^s }(\Omega_1)}\\
			&=\left[\left.\left(\dbar\dbar^* N\wt{f}\right)\right\vert_{\Omega_1}\right]_{H^{p,q}_{\mathrm{ext},W^s }(\Omega_1)}\\
			&=0,
			\end{align*}
			where in the first step, note that $\ea f$, being the zero extension of $f\in Z^{p.q}_{c,W^s}(\Omega_2)$, is automatically in $Z^{p,q}_{\mathrm{ext},W^s}$,
			and  in the second step, we have used the fact that $\dbar\ea f=0$ on $\Omega_1$, and hence $\dbar\wt{f}=0$ on $U$.  When $q=0$, this is a trivial consequence of Proposition~\ref{prop-maxrealproperties}, part~\ref{part-6}.}
		
	\end{enumerate}	
\end{proof}
Lemma~\ref{lem-modi2} above is equivalent to the fact that  in the following  diagram the triangle
and the rectangle both commute:
\begin{equation}\label{diag-trapezoid}
\begin{tikzcd}
%H^{p,q}_{c,W^s}(\Omega_2) \arrow{r}{0}\arrow[swap]{dr}{0}
& H^{p,q}_{\mathrm{ext}, W^{s}}(\Omega_1)\arrow{r}{(\rhoa_*)^{p,q}}
&H^{p,q}_{\mathrm{ext},W^s}(\Omega)\arrow{r}{c^{p,q}_\mathrm{A}}& H^{p,q+1}_{c,W^s}(\Omega_2) \\
& H^{p,q}_{W^{s+1}}(\Omega_1)\arrow[u, "(i_*^{\Omega_1})^{p,q}" ']  \arrow{r}{R^{p,q}_*} & H^{p,q}_{W^{s+1}}(\Omega)\arrow{u}{i^{p,q}_*} \arrow[swap]{ur}{\lambda^{p,q}}&
\end{tikzcd}
\end{equation}
Combining this with $\ea_*=0$ we see that the following sequence is exact:
\[	 \cdots \xrightarrow{{R}^{p,q-1}_*} H^{p,q-1}_{W^{s+1}}(\Omega)
\xrightarrow{\lambda^{p,q-1}}H^{p,q}_{c,W^s}(\Omega_2) \xrightarrow{0}H^{p,q}_{W^{s+1}}(\Omega_1) \xrightarrow{{R}^{p,q}_*}H^{p,q}_{W^{s+1}}(\Omega)
\xrightarrow{\lambda^{p,q}}\cdots\]
Therefore for each $q$, the map $R^{p,q}_*$ is injective, and the map $\lambda^{p,q}$ is surjective. It follows that  \eqref{eq-short-main} is exact for each $q$.
We already know that the map
$R_*$ is continuous, being induced by a continuous map of cochain complexes. The continuity of the map $\lambda$ follows from the formula
\eqref{eq-lambdadef}.

\subsection{Preliminaries for  Theorem~\ref{thm-long1}} \label{sec-prelim-long1}
\subsubsection{Sobolev realizations of the $\dbar$-operator}\label{sec-realizationdef}
Now we construct a long exact sequence associated to an annulus which relates the function theory of the annulus with that of its hole and its envelope.
An immediate consequence of our result is Corollary~\ref{cor-long-c} of the introduction, and in particular the very important exact sequence \eqref{eq-long-hole-min},
which encompasses many of the results about  $L^2$-estimates on annuli as found in \cite{lishaw, FLS} and earlier works cited there.

Let $D$ be a  domain in $\cx^n$.  By the  \emph{ domain of a
	$W^s$-realization $\db$ of the $\dbar$-operator on $D$}, we mean a collection of linear subspaces $A^{p,q}_{\db,W^s}(D)\subset W^s_{p.q}(\ol{D})$, where $0\leq p, q \leq n$, such that

\begin{enumerate}
	\item $ A^{p,q}_{c,W^s}(D)\subset A^{p,q}_{\db,W^s}(D) \subset A^{p,q}_{W^s}(D).$
	\item  if $f\in A^{p,q}_{\db,W^s}(D)$ then $\dbar f \in A^{p,q+1}_{\db,W^s}(D)$, with the derivative taken in the sense of distributions.
	\item for each $\phi\in \mathcal{C}^\infty(\ol{D})$, if $f\in A^{p,q}_{\db,W^s}(D)$ then $\phi f \in A^{p,q}_{\db,W^s}(D)$.
\end{enumerate}
Then the $\dbar$-operator acting on  $A^{p,q}_{\db,W^s}(D)$ in the sense of distributions is said to be a $W^s$-realization of
$\dbar$. It is clear that the three realizations $A^{p,*}_{W^s}(D), A^{p,*}_{c,W^s}(D)$  and $A^{p,*}_{\mathrm{ext},W^s}(D)$ of Section~\ref{sec-three-realizations}
satisfy the conditions above. As in section~\ref{sec-cochains}, we obtain a cochain sequence associated to the realization, and a corresponding cohomology group $H^{p,q}_{\db,W^s}(D)$.

\subsubsection{Mixed realizations}\label{sec-mixed}

Let $\db$ be a $W^s$-realization of $\dbar$ on the envelope $\Omega_1$ of the annulus $\Omega=\Omega_1\setminus\ol{\Omega_2}$.
We define a \emph{mixed realization} on $\Omega$ which coincides with $\db$ along $b\Omega_1$ and  with the minimal $W^s$-realization along $b\Omega_2$ in the following way.
Let $A^{p,q}_{(\db,c),W^s}(\Omega)$ consist of all $(p,q)$-currents $u$ on $\Omega$ of the form
\begin{equation}\label{eq-urep}
u=f|_\Omega + h,
\end{equation}
where
\[ \begin{cases} f\in A^{p,q}_{\db,W^s}(\Omega_1), f\equiv 0 \text{ in a neighborhood of } \ol{\Omega_2},\\
%g \in A^{p,q}_{\db_2,W^s}(\Omega), g \equiv 0 \text{ near } b\Omega_1\\
h\in A^{p,q}_{c, W^s}(\Omega).
\end{cases}\]
It is easily verified that the three conditions in section~\ref{sec-realizationdef} are  satisfied by $A^{p,q}_{(\db,c),W^s}(\Omega)$. Further, if
$f\in A^{p,q}_{(\db,c),W^s}(\Omega)$  then (i) if $\phi\in \mathcal{C}^\infty(\ol{\Omega})$ is such that $\phi$ vanishes near $b\Omega_2$, then the product $\phi f\in A^{p,q}_{\db,W^s}(\Omega_1)$, where $\phi f$ is assumed to be extended by zero in $\Omega_2$
and (ii) if $\psi \in \mathcal{C}^\infty(\ol{\Omega})$ is such that $\psi$ vanishes near $b\Omega_1$, then ${\psi} f\in A^{p,q}_{c,W^s}(\Omega) .$

Let
\[  \eh: A^{p,*}_{(\db,c),W^s}(\Omega)\to A^{p,*}_{\db,W^s}(\Omega_1)\]
be the \emph{zero-extension} operator defined in the following way. For  $u\in A^{p,q}_{(\db,c),W^s}(\Omega)$ represented as in \eqref{eq-urep},
we let
\begin{equation}\label{eq-ehdef}
\eh(u) = f + (\zs h)|_{\Omega_1},
\end{equation}
where $\zs h\in A^{p,q}_{W^s}(\cx^n)$ is the zero extension of $h$ to $\cx^n$ (see \eqref{eq-zeroext}). It is not difficult to see that $\eh$ is defined independently of the representation \eqref{eq-urep}, and $\eh$ is a
continuous cochain map.

%%%%%%%%%%%%%%%%%%%%%%%%%%%%%%%%%%%

\subsubsection{Definitions of maps.}\label{sec-holerest}
Suppose  that $\Omega=\Omega_1\setminus \ol{\Omega}_2$ is a bounded  annulus in which the hole $\Omega_2$ is Lipschitz, and suppose further
that we are given a $W^s$-realization $\db$ of the $\dbar$-operator $ A^{p,*}_{\db,W^s}(\Omega)$ .
Let $ A^{p,q}_{(\db,c),W^s}(\Omega)$ be the mixed realization  on $\Omega$ which coincides with the given realization $\db$ along $b\Omega_1$ and with the  	minimal realization along $b\Omega_2$ .
We define some maps:

\begin{enumerate}[wide, labelwidth=!, labelindent=0pt]
	\item From \eqref{eq-ehdef}, we obtain an induced map at the level of cohomology:
	\[ (\eh_*)^{p,q}: H^{p,q}_{(\db,c),W^s}(\Omega)\to H^{p,q}_{\db,W^s}(\Omega_1),\]
	which is continuous, since $\eh$ is (see Proposition~\ref{prop-ind-cont})
	
	\item Let $U$ be a bounded pseudoconvex domain containing $\ol{\Omega}_2$,
	and let $N$ denote the $\dbar$-Neumann operator of $U$. Define the \emph{modified restriction map}
	\[ S^{p,q}: H^{p,q}_{\db,W^s}(\Omega_1)\to H^{p,q}_{W^{s+1}}(\Omega_2)\]
	which is given for
	$g\in Z^{p,q}_{\db,W^s}(\Omega_1)$  by
	\begin{equation}\label{eq-sexpression}
	S^{p,q}\left([g]_{H^{p,q}_{\db,W^s}(\Omega_1)}\right)		=  \left[\left. \left(\dbar^* N \dbar (\wt{\chi\cdot g})\right)\right\vert_{\Omega_2} \right]_{H^{p,q}_{W^{s+1}}(\Omega_2)}
	\end{equation}
	where $\chi\in \mathcal{C}^\infty_0(\Omega_1)$ is a compactly supported smooth function such that $\chi\equiv 1$ in a neighborhood of $\ol{\Omega}_2$,
	and $\wt{\chi\cdot  g}$ is the form on $U$ obtained by extending the compactly supported form $\chi\cdot g$ by zero on $U\setminus \Omega_1$, if this set is nonempty.  {Note that $\dbar(\wt{\chi\cdot g})=\wt{\dbar\chi\wedge g}\in W^s_{p,q+1}(U)$, so the interior regularity of the $\dbar$-Neumann problem will guarantee that $\left. \left(\dbar^* N \dbar (\wt{\chi\cdot g})\right)\right\vert_{\Omega_2}\in W^{s+1}_{p,q}(\Omega_2)$}.
	It will follow
	from the proof of Theorem~\ref{thm-long1} below that $S^{p,q}$ is  well-defined and the definition \eqref{eq-sexpression} is independent of the choice of the pseudoconvex neighborhood
	$U$ of the hole $\ol{\Omega}_2$ and the cutoff $\chi$.
	
	\item
	
	We  introduce a \emph{modified connecting homomorphism}
	\[ \ell^{p,q}: H^{p,q}_{W^{s+1}}(\Omega_2)\to H^{p,q+1}_{(\db,c), W^s}(\Omega)\]
	which is given  for
	$f\in Z^{p,q}_{W^{s+1}}(\Omega_2)$ by
	\begin{equation}\label{eq-lpqdef}
	\ell^{p,q}\left(\cls{f}{H^{p,q}_{W^{s+1}}(\Omega_2)}\right)=\cls{\dbar(\chi\cdot Ef)}{H^{p,q+1}_{(\db,c),W^s}(\Omega)},
	\end{equation}
	where $\chi\in \mathcal{D}(\Omega_1)$ is such that $\chi\equiv 1$ in a neighborhood of $\Omega_2$, and $E:W^{s+1}(\Omega_2)\to W^{s+1}(\cx^n)$ is the extension operator acting coefficientwise on
	forms.
\end{enumerate}

%	
%	 It is easy to see that
%the maps
%\[ \rhoh: (A^{p,*}_{\db,W^s}(\Omega_1), \dbar)\to (A^{p,q}_{\mathrm{ext},W^s}(\Omega_2),\dbar)\]
%and
%\[  \eh: (A^{p,*}_{(\db,c),W^s}(\Omega),\dbar)\to (A^{p,*}_{\db,W^s}(\Omega_1),\dbar)\]
%are continuous cochain morphisms of cochain complexes of inner-product spaces. Indeed, they are cochain morphisms since the differential operator is the same in all
%of them, namely, the $\dbar$-operator acting in the sense of distributions on currents. The map is continuous, as can be seen from the fact that all these are inner-product spaces
%in the same inner product which generates the graph norm.
%
%To write down the first of our relations between cohomologies of the annulus, hole and envelope, we introduce some maps.

\subsection{A long exact sequence associated to annuli}\label{sec-long1}

Now we can state and prove the second main result of this paper, using the notions introduced in the preceding Section~\ref{sec-prelim-long1}.

\begin{thm} \label{thm-long1} Let $\Omega=\Omega_1\setminus \ol{\Omega}_2\subset \cx^n$  be an annulus.
	Let $\db$ be a realization of the $\dbar$-operator on $W^s_{p,*}(\Omega_1)$, where $s$ is an integer, and $0\leq p\leq n$.
	With notation introduced as above, the following sequence of semi-inner-product spaces  and continuous linear maps is exact:
	\begin{equation}\label{eq-long-hole} \cdots \xrightarrow{{S}^{p,q-1}} H^{p,q-1}_{W^{s+1}}(\Omega_2)
	\xrightarrow{\ell^{p,q-1}}H^{p,q}_{(\db,c),W^s}(\Omega) \xrightarrow{(\eh_*)^{p,q}}H^{p,q}_{\db,W^s}(\Omega_1) \xrightarrow{{S}^{p,q}}H^{p,q}_{W^{s+1}}(\Omega_2)
	\xrightarrow{\ell^{p,q}}\cdots\end{equation}
\end{thm}			

\subsubsection{Step 1 of proof: From short to long exact sequence}

Define  for $f$ a current on
$\Omega_1$:
\[ \rhoh(f) =f|_{\Omega_2}\quad \text{(restriction to the hole)}.\]
We claim that the short exact sequence  of inner-product spaces and continuous cochain maps
\begin{equation}\label{eq-sesq2}
0\to A^{p,*}_{(\db,c),W^s}(\Omega)\xrightarrow{\eh} A^{p,*}_{\db,W^s}(\Omega_1)\xrightarrow{\rhoh}A^{p,*}_{\mathrm{ext}, W^s}(\Omega_2)\to 0.
\end{equation}
is exact. It is clear from the definition that $\eh$ is an injective continuous cochain morphism.
By definition $\rhoh$ is a continuous mapping of inner-product spaces.  Part~7 of Proposition~\ref{prop-maxrealproperties} shows that the map
$\rhoh: A^{p,*}_{c, W^s}(\Omega_1)\to A^{p,*}_{\mathrm{ext}, W^s}(\Omega_2)$ is surjective. Since by hypothesis, $A^{p,*}_{c, W^s}(\Omega_1) \subset A^{p,*}_{\db,W^s}(\Omega_1)$
it follows that $\rhoh$ is surjective onto  $ A^{p,*}_{\mathrm{ext}, W^s}(\Omega_2)$.
To see exactness at $A^{p,*}_{\db,W^s}(\Omega_1)$ we simply note that
\[ \ker \rhoh =\{ f\in A^{p,*}_{\db, W^s}(\Omega_1)\mid f|_{\Omega_2}=0\}= \img \eh.\]

Now, again using  a well-known result  in algebra (see, e.g.,  \cite[Chapter XX, Theorem 2.1]{lang}),  the short exact sequence \eqref{eq-sesq2} gives rise to a long exact sequence of semi-inner-product spaces and linear maps:
\begin{equation}\label{eq-long1}
\cdots \xrightarrow{(\rhoh_*)^{p,q-1}} H^{p,q-1}_{\mathrm{ext}, W^s}(\Omega_2)
\xrightarrow{c_\mathrm{H}^{p,q-1}}H^{p,q}_{(\db,c),W^s}(\Omega) \xrightarrow{(\eh_*)^{p,q}}H^{p,q}_{\db,W^s}(\Omega_1) \xrightarrow{(\rhoh_*)^{p,q}}H^{p,q}_{\mathrm{ext},W^s}(\Omega_2)
\xrightarrow{c_\mathrm{H}^{p,q}}\cdots\end{equation}
where $\rhoh_*$ and $\eh_*$ are the maps induced
on the cohomology by the maps $\rhoh$ and $\eh$, and therefore are continuous by Proposition~\ref{prop-ind-cont}, and $c_\mathrm{H}$  is the ``connecting
homomorphism", a linear mapping
\[ c_\mathrm{H}^{p,q}: H^{p,q}_{\mathrm{ext}, W^s}(\Omega_2)\to H^{p,q+1}_{(\db, c)}(\Omega)\]
defined by the formula (with  $f\in Z^{p,q}_{\mathrm{ext},W^s}(\Omega_2) $),
\begin{equation}\label{eq-cdef} c_\mathrm{H}^{p,q}\left([f]_{H^{p,q}_{\mathrm{ext},W^s}(\Omega_2)}\right)=  \left[\left.{\dbar g}\right\vert_\Omega \right]_{H^{p,q+1}_{(\db,c)}(\Omega)},  \end{equation}
where  the element $g\in {A^{p,q}_{\db,W^s}}(\Omega_1)$ is chosen such that ${\rhoh(g)}= f.$
Further,  $c_\mathrm{H}$ is well-defined independently of the choice of the ``lift" $g$ of $f$.

Recall that the map
$i_*: H^{p,q}_{W^{s+1}}(\Omega_2)\to H^{p,q}_{\mathrm{ext},W^s}(\Omega_2)$  of \eqref{eq-istar}
is a continuous isomorphism of semi-inner-product spaces,
and is in fact the map at the cohomology level induced by
the inclusion map of cochain complexes $i:   A^{p,q}_{W^{s+1}}(\Omega_2) \hookrightarrow A^{p,q}_{\mathrm{ext}, W^s}(\Omega_2)
$. We have the following representations of the maps
$\ell^{p,q}$ and $S^{p,q}$ introduced in \eqref{eq-lpqdef} and \eqref{eq-sexpression} above.

\begin{lem} \label{lem-lsrep} We have
	\begin{equation}
	\label{eq-lrep}
	\ell^{p,q}= c_\mathrm{H}^{p,q}\circ i_*^{p,q}\end{equation}
	and
	\begin{equation}
	\label{eq-Srep}
	S^{p,q}= (i_*^{p,q})^{-1}\circ (\rhoh_*)^{p,q}.\end{equation}

\end{lem}
\begin{proof}
	Let $f\in Z^{p,q}_{W^{s+1}}(\Omega_2)$.  Then we have
	\begin{align*}
	(c^{p,q}_\mathrm{H}\circ  i_*)\left( [f]_{H^{p,q}_{W^{s+1}}(\Omega_2)}\right)  &= c^{p,q}_\mathrm{H}\left( i_*\left([f]_{H^{p,q}_{W^{s+1}}(\Omega_2)}\right)\right)\\&=c^{p,q}_\mathrm{H}\left( [f]_{H^{p,q}_{\mathrm{ext},W^s}(\Omega_2)}\right)=\left[{\left.\dbar g\right\vert_\Omega}  \right]_{H^{p,q+1}_{(\db,c)}(\Omega)}=\ell^{p,q}\left( [f]_{H^{p,q}_{W^{s+1}}(\Omega_2)}\right) ,
	\end{align*}
	where $g$ is an element of ${A^{p,q}_{\db,W^s}}(\Omega_1)$ which is mapped by
	${\rhoh}$ onto $f$. Since $f\in W^{s+1}_{p,q}(\Omega_2)$, we can take $g=\chi \cdot Ef$ with $\chi$ and $E$ as in \eqref{eq-lpqdef}.
	
	Now let $g\in Z^{p,q}_{\db, W^s}(\Omega_1)$. Then we have
	\begin{align*}
	\left((i_*^{p,q})^{-1}\circ (\rhoh_*)^{p,q}\right)\left([g]_{H^{p,q}_{\db, W^s}(\Omega_1)}\right)&=	 (i_*^{p,q})^{-1} \left(\left[\left.g\right\vert_{\Omega_2}\right]_{H^{p.q}_{\mathrm{ext},W^s}(\Omega_2)}\right)\\
	&=\left[\left.\left(\dbar^* N \dbar h\right)\right\vert_{\Omega_2} \right]_{H^{p,q}_{W^{s+1}}(\Omega_2)},
	\end{align*}
	where in the last line we have used the representation \eqref{eq-istarinv} of $(i_*^{p,q})^{-1}$. Here $N$ is the $\dbar$-Neumann operator
	of a bounded pseudoconvex neighborhood $U$ of $\ol{\Omega}_2$, and $h$ is an extension of $\left.g\right\vert_{\Omega_2}$ as an element of
	$A^{p,q}_{c,W^s}(U)$. For the cutoff $\chi$ in \eqref{eq-sexpression},  which has support in $\Omega_1$ and is 1 near $\ol{\Omega}_2$
	we can take $h= \wt{\chi\cdot g}$, thus establishing \eqref{eq-Srep}.
\end{proof}
\subsubsection{End of the proof of Theorem~\ref{thm-long1}}
The representation \eqref{eq-lrep} of the map $\ell^{p,q}$ shows that the definition  \eqref{eq-lpqdef} does not depend on the choice of the  cutoff $\chi$.
Since  the map
from $ Z^{p,q}_{W^{s+1}}(\Omega_2)$ to $H^{p,q+1}_{(\db,c),W^s}(\Omega)$
given by $
f\mapsto  \cls{\left.\dbar(\chi\cdot Ef)\right\vert_\Omega}{H^{p,q+1}_{(\db,c)}(\Omega)}
$
is continuous,  it follows by the universal property of the quotient topology (see diagram~\ref{diag-quotient}) that the induced map 	\[
\ell^{p,q}\left(\cls{f}{H^{p,q}_{W^{s+1}}(\Omega_2)}\right)=\cls{\dbar(\chi\cdot Ef)|_\Omega}{H^{p,q+1}_{(\db,c)}(\Omega)}
\] is  also continuous.

The representation \eqref{eq-Srep}  shows that $S^{p,q}$ is defined independently of the choice of the pseudoconvex open set $U$ and the cutoff $\chi$.
To see continuity of $S^{p,q}$ notice that the map
from $ Z^{p,q}_{\db,W^s}(\Omega_1)$  to $H^{p,q}_{W^{s+1}}(\Omega_2)$
given by
\[	g\mapsto  \left[\left. \left(\dbar^* N \dbar (\wt{\chi\cdot g})\right)\right\vert_{\Omega_2} \right]_{H^{p,q}_{W^{s+1}}(\Omega_2)}\]
is easily seen to be continuous (using the interior regularity of the canonical solution operator $\dbar^* N$), and therefore, the map $S^{p,q}$ induced
by this map is also continuous, again by an appeal to the universal property of the quotient in diagram~\ref{diag-quotient}.

To complete the proof of Theorem~\ref{thm-long1}, we now see  from Lemma~\ref{lem-lsrep} that each of the two triangles in the following diagram commutes:

\[
\begin{tikzcd}
H^{p,q}_{\db,W^s}(\Omega_1) \arrow{r}{(\rhoh_*)^{p,q}}\arrow[swap]{dr}{S^{p,q}} & H^{p,q}_{\mathrm{ext},W^s}(\Omega_2)\arrow{r}{c^{p,q}_\mathrm{H}} &H^{p,q+1}_{(\db,c),W^s}(\Omega) \\
& H^{p,q}_{W^{s+1}}(\Omega_2)\arrow{u}{i^{p,q}_*}\arrow[swap]{ur}{\ell^{p,q}}&
\end{tikzcd}
\]

This combined with \eqref{eq-long1} shows that the sequence \eqref{eq-long-hole} is exact. This completes the proof of Theorem~\ref{thm-long1}.

\subsection{Proof of Corollary~\ref{cor-long-c}}\label{sec-longcor}
Using Theorem~\ref{thm-long1} it is easy to complete the proof of the corollary stated in the introduction, which follows on letting  the realization $\db$ on $\Omega_1$ in Theorem~\ref{thm-long1} be the minimal Sobolev realization with domain $A^{p,q}_{c,W^s}(\Omega_1)$, and noting that the resulting mixed realization in the annulus is the minimal realization with domain $A^{p,q}_{c,W^s}(\Omega)$.
\section{Applications to function theory on annuli}\label{sec-applications}

\subsection{Duality for Sobolev cohomology}

As a first preliminary to applying Theorems~\ref{thm-short} and \ref{thm-long1} to concrete questions, we discuss generalization to Sobolev spaces of well-known duality phenomena for
the $L^2$-Dolbeault cohomology (see \cite{serre}).

Given Banach spaces $X$ and $Y$, we call a continuous bilinear map $\beta:X\times Y\rightarrow\mathbb{C}$ \emph{perfect} if for each continuous functional $\phi\in X^*$ there exists a unique $y\in Y$ such that $\phi(x)=\beta(x,y)$ and for each continuous linear functional $\psi\in Y^*$ there exists a unique $x\in X$ such that $\psi(y)=\beta(x,y)$.  Note that this bilinear map identifies $X^*$ with $Y$ and $Y^*$ with $X$, which implies that $X$ and $Y$ are reflexive.

Let $D$ be a bounded domain in $\cx^n$, and let $0\leq p,q \leq n$. Consider the natural bilinear map
\[ \mathcal{C}^\infty_{p,q}(\ol{D})\times \mathcal{D}^{n-p,n-q}(D)\to\cx,\]
given by
\begin{equation}
\label{eq-integral-pairing}	f,g \mapsto \int_D f\wedge g.
\end{equation}
Let $s$ be an integer. Since $\mathcal{C}^\infty_{p,q}(\ol{D})$ is dense in $W^s_{p,q}(\ol{D})$ and $\mathcal{D}^{n-p,n-q}(D)$ is dense in $(W^{-s}_0)_{p,q}(\ol{D})$ (see part (2) of Proposition~\ref{prop-maxrealproperties}), it follows that the bilinear map \eqref{eq-integral-pairing} extends to a separately continuous bilinear map
\begin{equation}
\label{eq-sobolev-pairing} W^s_{p,q}(\ol{D})\times (W^{-s}_0)_{n-p,n-q}(\ol{D}) \to \cx,
\end{equation}
which in fact is continuous by a standard application of the uniform boundedness principle. We continue to denote this  pairing of Hilbert spaces of currents by
the integral notation \eqref{eq-integral-pairing}. By construction this pairing is perfect.

We can think of the $W^s$-Sobolev realization of the $\dbar$-operator with domain $A^{p,q}_{W^s}(D)$ as a densely-defined closed operator
\begin{equation}
\label{eq-dbar1} \dbar_{W^s}:W^s_{p,q-1}(\ol{D})\dashrightarrow W^s_{p,q}(\ol{D}).
\end{equation}

The key to the duality theory in the $\dbar$ problem is the following:
\begin{prop}
	Under the identification of dual spaces given by \eqref{eq-integral-pairing}, the transpose of the operator $\dbar_{W^s}$ of \eqref{eq-dbar1} is  the unbounded operator
	\[(-1)^{p+q}\dbar_{c, W^{-s}}: (W^s_0)_{n-p, n-q-1}(\ol{D}) \dashrightarrow  (W^s_0)_{n-p, n-q}(\ol{D})\]
	with domain $  A^{n-p,n-q-1}_{c, W^{-s}}(D)$.
\end{prop}
In other words, the minimal $W^{-s}$-realization corresponds to the \emph{dual} co-chain complex (up to a sign depending on degree)  of the $W^s$-realization in the language of \cite{meichichristine}.
\begin{proof}
	 We can identify $(W^s_{p,q}(\ol{D}))' $ with $(W^{-s}_0)_{n-p,n-q}(\ol{D})$,  and identify $(W^s_{p,q-1}(\ol{D}))'$ with $(W^{-s}_0)_{n-p,n-q+1}(\ol{D})$ via the  pairing 	\eqref{eq-integral-pairing}. To determine the domain of the domain of definition of $(\dbar_{W^s})'$ as a subspace of $(W^{-s}_0)_{n-p,n-q}(\ol{D})$ under this identification, note that
	this domain consists of those $g\in (W^{-s}_0)_{n-p,n-q}(\ol{D})$, for which the map from $A^{p,q-1}_{W^s}(D)$ to $\cx$ given by
	\begin{equation}
	\label{eq-map}  f\mapsto \int_{D}\dbar f \wedge g
	\end{equation}
	extends to a bounded linear functional on $W^s_{p,q-1}(\ol{D})$.
	Notice now that if $f\in \mathcal{C}^\infty_{p,q-1}(\ol{D})$ and $g\in \mathcal{D}^{n-p,n-q}(D)$ then we have
	\[\int_D \dbar f \wedge g = \int_D \left(\dbar(f\wedge g) - (-1)^{p+q-1}f \wedge \dbar g\right)= (-1)^{p+q}\int_D f\wedge \dbar g.\]
	Therefore, using the perfectness of the pairing \eqref{eq-integral-pairing} we see that
	\eqref{eq-map} extends to an element of $(W^s_{p,q-1}(\ol{D}))'$ if and only if $\dbar g \in (W^{-s}_0)_{n-p,n-q+1}(\ol{D}) $
	i.e., $g\in A^{n-p,n-q}_{c, W^{-s}}(D)$.

	Since $\mathcal{C}^\infty_{p,q-1}(\ol{D})$ is dense  in $A^{p,q}_{W^s}(D)$, and $\mathcal{D}^{n-p,n-q}(D)$  is dense in $A^{n-p,n-q}_{c, W^{-s}}(D)$ it follows that
	under the pairing \eqref{eq-integral-pairing}, for $f\in A^{p,q}_{W^s}(D)$ and $g\in A^{n-p,n-q-1}_{c, W^{-s}}(D)$ we have
	\[ \int_D \dbar f \wedge g = (-1)^{p+q}\int_D f\wedge \dbar g.\]
	Therefore,
	The result follows.
\end{proof}
Recall the definitions of the indiscrete part and the reduced form of a semi-inner-product space (see \eqref{eq-indiscrete} and  \eqref{eq-reduced} above). The indiscrete part and the reduced form of a cohomology group are called the \emph{indiscrete cohomology} and the \emph{reduced cohomology}, respectively.

The pairing \eqref{eq-sobolev-pairing} gives rise to a pairing of reduced cohomologies
\begin{equation}\label{eq-serre1}
\Red H^ {p,q}_{W^{{s}}}(D)\times \Red H^{n-p,n-q}_{c, W^{-s}}(D)\to \cx
\end{equation}
called the \emph{Serre pairing},  given by
\begin{equation}\label{eq-serre2}
\cls{f}{}, \cls{g}{} \mapsto  \int_\Omega \cls{f}{}\wedge \cls{g}{}:= \int_D f\wedge g, \quad  f\in Z^ {p,q}_{W^{{s}}}(D) , g \in  Z^{n-p,n-q}_{c, W^{-s}}(D),
\end{equation}
where the integral in \eqref{eq-serre2} is understood in the same sense as in \eqref{eq-integral-pairing}, i.e. as a limit. It is not difficult to see that this mapping is well-defined and a continuous bilinear mapping of Hilbert spaces.
Using this pairing we can state the following result, analogs of which may be found in \cite{serre, laufer, serreduality,meichichristine}:
\begin{prop}
	\label{prop-serre-laufer}
Let $D$ be a bounded Lipschitz domain in $\cx^n$, let $0\leq p,q \leq n$ and $s$ be an integer. Then
\begin{enumerate}[wide]	\item $\Ind H^{p,q}_{W^s}(D)=0$ if and only if $\Ind H^{n-p,n-q+1}_{c, W^{-s}}(D)=0$.
	\item the {Serre pairing} \eqref{eq-serre1} given by  \eqref{eq-serre2}
	is perfect, so
	$\Red H^{p,q}_{W^s}(D)$ and $\Red H^{n-p,n-q}_{c, W^{-s}}(D))$ are duals of each other via the Serre pairing.

	\end{enumerate}
\end{prop}
\begin{proof}
	Part (1) is equivalent to the statement that $H^{p,q}_{W^{s}}(D)$ is Hausdorff
	if and only if $H^{n-p,n-q+1}_{c, W^{-s}}(D)$ is Hausdorff. Further,  $H^{p,q}_{W^{s}}(D)$ is Hausdorff if and only if the  operator $\dbar_{W^s}: W^s_{p,q-1}(\ol{D})\dashrightarrow W^s_{p,q}(\ol{D})$ has closed range,
	but this is equivalent to the fact that the transposed operator $\dbar_{c, W^{-s}}: (W^s_0)_{n-p, n-q-1}(\ol{D}) \dashrightarrow  (W^s_0)_{n-p, n-q}(\ol{D})$ has closed range, i.e., $H^{n-p,n-q+1}_{c, W^{-s}}(D)$ is Hausdorff.
	
	Part (2): let $\gamma: \Red H^{p,q}_{W^s}(D)\to \cx$ be a bounded linear functional. We need to show that there is a unique cohomology class
	$\theta\in \Red H^{n-p,n-q}_{c, W^{-s}}(D) $ such that for each cohomology class $\alpha \in \Red H^{p,q}_{W^s}(D)$ we have
	\[ \gamma(\alpha)= \int_D \theta\wedge \alpha. \]

	To see uniqueness of such $\theta$, suppose that there are $\theta_1, \theta_2 \in \Red H^{n-p,n-q}_{c, W^{-s}}(D) $ such that for each cohomology class $\alpha \in \Red H^{p,q}_{W^s}(D)$ we have $\gamma(\alpha)= \int_D \theta_1\wedge \alpha = \int_D \theta_2\wedge \alpha$. If $\theta_1=[g_1]$ and
	$\theta_2=[g_2]$, this means that for each $f\in  Z^{p,q}_{W^{{s}}}(D)$ we have $\int_Df \wedge (g_1-g_2)=0$. Since $f\in \ker \dbar_{W^{{s}}}$, this means
	that $g_1-g_2\in \ol{\mathrm{range}\,\dbar_{c, W^{-s}} }$. Therefore
	$\theta_1=[g_1]=\theta_2=[g_2]$ in $\Red H^{n-p,n-q}_{c, W^{-s}}(D) $ .

	To prove existence of $\theta$, let $B_s: W^{{s}}_{p,q}(D)\to \ker \dbar_{W^{{s}}}\cap W^{{s}}_{p,q}(D)$ denote the orthogonal projection (with respect to the $W^{{s}}$ inner product) onto the space of $\dbar$-closed forms. Consider the linear map $\wt{\gamma}: W^{{s}}_{p,q}(D)\to \cx$ given by
	\[ \wt{\gamma}(f)= \gamma([B_s f]),\]
	which is obviously continuous. We can find a {$g\in (W^{-s}_0)_{n-p,n-q}(D)$} such that
	$\wt{\gamma}(f)= \int_D g\wedge f$. We notice that this $g$ in fact lies in {$Z^{n-p,n-q}_{c, W^{-s}}(D)$}: if $h\in \mathcal{D}_{p,q-1}(D)$ is a test-form, then pairing a current against a test form we have
	\[\int_D \dbar g\wedge h  =(-1)^{p+q+1}\int_D g \wedge \dbar h = (-1)^{p+q+1}\wt{\gamma}(\dbar h)= (-1)^{p+q+1}\gamma([B_s \dbar h]) =0.  \]
	Therefore, if $\theta=[g]\in \Red H^{n-p,n-q}_{c, W^{-s}}(D)$, we have $\gamma(\alpha)=\int_D\theta\wedge \alpha$ for each class $\alpha\in \Red H^{p,q}_{W^s}(D)$.
	
	Similarly, if $\delta:\Red H^{n-p,n-q}_{c, W^{-s}}(D)\to \cx$ is a continuous linear functional,  a similar argument shows that there is a unique
	$\phi\in \Red H^{p,q}_{W^s}(D)$ such that for each $\beta\in \Red H^{n-p,n-q}_{c, W^{-s}}(D)$ we have
	\[ \delta(\beta)= \int_D \beta\wedge \phi. \]
\end{proof}

\subsection{Vanishing theorems in pseudoconvex domains}
In order to apply the exact sequences, we will need to use vanishing results
for the cohomology, which we will now recall. To state the results succinctly let us introduce the following definition: let $s\geq 0$ be an integer. We will say that a bounded domain $D\subset\cx^n$ is
\emph{sufficiently smooth} (for the integer $s$) if
\begin{enumerate}
	\item $s\geq 2$ and the boundary of $D$ is of class $\mathcal{C}^{s+1}$.
	\item $s=1$ and the boundary of $D$ is of class $\mathcal{C}^{1,1}$.
	\item $s=0$ and no conditions are imposed on the boundary.
\end{enumerate}

The following result summarizes the basic facts about the $\dbar$-problem on pseudoconvex domains. For $s=0$, this goes back to \cite{hormander1965}, and for $s\geq 1$
to \cite{kohn73}. The more refined boundary conditions were obtained in \cite{harrington2009} (for $s\geq 2$) and \cite{DebrajPhil} for $s=1$.

\begin{thm}\label{thm-hormander-kohn} Let   $D$ be a bounded pseudoconvex domain in $\cx^n$, and let $s\geq 0$ be an integer. Then  $H^{p,0}_{W^s}(D)$ is
	the infinite dimensional Hilbert space of holomorphic $p$-forms with $W^s$ coefficients. Further, if $D$ is sufficiently smooth in the above sense for the integer $s$,
	we have $H^{p,q}_{W^s}(D)=0$ if $q\geq 1$.
\end{thm}

From this we can deduce the following, which is  well-known for $s=0$ (see \cite{serre}):
\begin{cor}
	\label{cor-l2coh-compact}
	Let $D\subset\cx^n$
	be a bounded pseudoconvex domain for $n\geq 1$ and let $s\geq 0$ be an integer {such that $D$ is sufficiently smooth} for $s$. For  $0\leq p \leq n$ we have
	\[
	H^{p,q}_{c, W^{-s}}(D)= \begin{cases}
	0 & \text{ if }0\leq q \leq n-1,\\
	\text{can be identified with {$(H^{n-p,0}_{W^{s}}(D))'$}} & \text{  if } q=n.
	\end{cases}
	\]
	It follows that $H^{p,n}_{c, {W^{-s}}}(D)$ is	{{Hausdorff} and infinite dimensional}.
\end{cor}
\begin{proof} The case $q=0$ is contained in part~\ref{part-6} of Proposition~\ref{prop-maxrealproperties}.
	Theorem~\ref{thm-hormander-kohn} combined with Proposition~\ref{prop-serre-laufer}  shows that the groups $H^{p,q}_{c,W^{-s}}(D)$ are all Hausdorff. Therefore, by Serre duality (Proposition~\ref{prop-serre-laufer})  the Serre pairing
	\[ H^{p,q}_{c, W^{-s}}(D)\times H^{n-p,n-q}_{W^s}(D)\to \cx,\]
	is perfect for all $p,q$. It follows that $H^{p,q}_{c, W^{-s}}(D)=0$ provided $H^{n-p,n-q}_{W^s}(D)=0$, i.e, $n-q \geq 1$ or equivalently
	$q \leq n-1$. For $q=n$ we see that the Serre pairing gives an identification of $H^{p,q}_{c, W^{-s}}(D)$ with $(H^{n-p,n-q}_{W^s}(D))'$.
\end{proof}

\subsection{An application of Theorem~\ref{thm-short}: envelopes with vanishing cohomology}\label{sec-short} In this section we consider the consequences of the short exact sequence \eqref{eq-short-main} when we assume that
in some degree $(p,q)$ we have vanishing of the $W^{s+1}$-cohomology of the envelope, i.e.,
\begin{equation}
\label{eq-envelope-vanishing}
H^{p,q}_{W^{s+1}}(\Omega_1)=0.
\end{equation}
\begin{enumerate}[wide]
	\item \emph{Assume that in an  annulus  \eqref{eq-envelope-vanishing} holds.}   From the exactness of  \eqref{eq-short-main}, we see that the map
	\begin{equation}
	\label{eq-lambda-bijection}
	\end{equation}
	\[\lambda^{p,q}:H^{p,q}_{W^{s+1}}(\Omega)\to H^{p,q+1}_{c, W^s}(\Omega_2)\]
	is a continuous bijection of semi-inner-product spaces. Notice that such a map does not necessarily have a continuous inverse if, for example, $H^{p,q+1}_{c, W^s}(\Omega_2)$ is indiscrete
	and $H^{p,q}_{W^{s+1}}(\Omega)$ is not indiscrete, i.e., has a nontrivial Hausdorff summand. It would be interesting to see if there are annuli for which the cohomologies have these properties.
	
	\item \emph{Assume that in an annulus  \eqref{eq-envelope-vanishing} holds, and  further that $H^{p,q+1}_{c, W^s}(\Omega_2)$ is Hausdorff.} Then since in \eqref{eq-lambda-bijection}, $\lambda^{p,q}$ is an injective continuous map, it follows that $H^{p,q}_{W^{s+1}}(\Omega)$ is also Hausdorff. Therefore $\lambda^{p,q}$ in \eqref{eq-lambda-bijection} is a continuous bijection of Hilbert spaces, and therefore is an isomorphism (i.e. has a bounded inverse) by a standard application of the closed-graph theorem.
	
	Thanks to Proposition~\ref{prop-serre-laufer}, if $\Omega_2$ has Lipschitz boundary,  the hypothesis that $H^{p,q+1}_{c, W^s}(\Omega_2)$ is Hausdorff is equivalent to the hypothesis that $H^{n-p,n-q}_{W^{-s}}(\Omega_2)$ is Hausdorff.
	
	\item \emph{Assume that in an annulus  \eqref{eq-envelope-vanishing} holds, and that $H^{n-p,n-q}_{W^{-s}}(\Omega_2)$ and $H^{n-p,n-q-1}_{W^{-s}}(\Omega_2)$ are both Hausdorff.} Then, as we remarked above, {$H^{p,q+1}_{W^{s}}(\Omega_2)$}
	is Hausdorff, and therefore  by Serre duality (Proposition~\ref{prop-serre-laufer}) the  pairing  \eqref{eq-serre2} is perfect and gives rise to an isomorphism
	\[ H^{p,q+1}_{c,W^s}(\Omega_2)\cong (H^{n-p,n-q-1}_{W^{-s}}(\Omega_2))'.\]
	Therefore, composing the map $\lambda^{p,q}$ of \eqref{eq-lambda-bijection} with the Serre pairing, we obtain a perfect pairing
	\begin{equation}
	\label{eq-shaw1} H^{p,q}_{W^{s+1}}(\Omega)\times H^{n-p,n-q-1}_{W^{-s}}(\Omega_2)\to \cx
	\end{equation}
	given by
	\begin{equation}
	\label{eq-shaw2} [f], [g]\mapsto \int_{\Omega_2}\lambda^{p,q}\left([f]\right)\wedge [g]= \int_{\Omega_2}\dbar Ef\wedge g ,
	\end{equation}
	where $f\in Z^{p,q}_{W^{s+1}}(\Omega) , g\in Z^{n-p,n-q-1}_{W^{-s}}(\Omega_2)$ and the integral in \eqref{eq-shaw2} is interpreted in a limiting sense as in
	the Serre pairing \eqref{eq-serre2}.
	
	\item The cases $s=0$ and $s=-1$ of the pairing \eqref{eq-shaw1} were noted by Shaw in \cite{meichiharmonic} for $q=n-1$, when $\Omega_1$ and $\Omega_2$ are both
	pseudoconvex. Here the pairing can be represented as a boundary integral:
	\begin{prop}\label{prop-shaw}
		Let $\Omega=\Omega_1\setminus\ol{\Omega}_2$ be an annulus where $\Omega_2$ has Lipschitz boundary,
		and suppose that $0\leq p,q\leq n$
		\begin{enumerate}[wide]
			
			\item ($s=0$) Suppose that $H^{p,q}_{W^{1}}(\Omega_1)=0$ and  the two groups $H^{n-p,n-q}_{L^2}(\Omega_2)$ and $H^{n-p,n-q-1}_{L^2}(\Omega_2)$ are both Hausdorff.  Then $H^{p,q}_{W^{1}}(\Omega)$ is Hausdorff and the pairing
			\[ H^{p,q}_{W^{1}}(\Omega)\times H^{n-p,n-q-1}_{L^2}(\Omega_2)\to \cx\]
			given by
			\begin{equation} \label{eq-boundary-integral1}
			[f], [g]\mapsto \int_{b\Omega_2}f\wedge P_{\Omega_2}g
			\end{equation}
			is perfect, where
			$ P_{\Omega_2}:Z^{n-p,n-q-1}_{L^2}(\Omega_2)\to \mathscr{H}^{n-p,n-q-1}_{L^2}(\Omega_2)$
			is the harmonic projection on $\Omega_2$.
			\item ($s=-1$) Suppose that $H^{p,q}_{L^2}(\Omega_1)=0$ and  the two groups $H^{n-p,n-q}_{W^1}(\Omega_2)$ and $H^{n-p,n-q-1}_{W^1}(\Omega_2)$ are both Hausdorff.   Then $H^{p,q}_{L^2}(\Omega)$ is Hausdorff and  the pairing
			\[ H^{p,q}_{L^2}(\Omega)\times H^{n-p,n-q-1}_{W^1}(\Omega_2)\to \cx\]
			given by
			\begin{equation} \label{eq-boundary-integral2}
			[f], [g]\mapsto \int_{b\Omega_2}P_\Omega f\wedge g
			\end{equation}
			is perfect, where
			$P_{\Omega}:Z^{p,q}_{L^2}(\Omega)\to \mathscr{H}^{p,q}_{L^2}(\Omega)$
			is the harmonic projection on $\Omega$.
		\end{enumerate}
	\end{prop}
\end{enumerate}

\textbf{Remarks:} \begin{enumerate}[wide]
	\item Recall that the \emph{ harmonic space} in degree $(p,q)$ of a domain $D\subset\cx^n$ is defined to be
	\[  \mathscr{H}^{p,q}_{L^2}(D)= \{f\in \dom (\dbar)\cap \dom (\dbar^*): \dbar f=\dbar^*f=0\},\]
	where $\dom(\dbar)= A^{p,q}_{L^2}(D)$  and $\dbar^*:L^2_{p,q}(D)\dashrightarrow L^2_{p,q-1}(D)$ is the Hilbert space adjoint of
	$\dbar: L^2_{p,q-1}(D)\dashrightarrow L^2_{p,q}(D)$.  This is a closed subspace of $Z^{p,q}_{L^2}(D)$, since $\dbar^*$ is a closed operator. If both $\dbar$ and $\dbar^*$ have closed range the harmonic projection
	$P_D: Z^{p,q}_{L^2}(D)\to \mathscr{H}^{p,q}_{L^2}(D)$ descends to an isomorphism $H^{p,q}_{L^2}(D)\to \mathscr{H}^{p,q}_{L^2}(D)$ (Hodge representation of cohomology).
	In particular, for each $f\in  Z^{p,q}_{L^2}(D)$, we have $[f]=[P_Df]$ in $H^{p,q}_{L^2}(D)$.
	\item The integrals \eqref{eq-boundary-integral1} and \eqref{eq-boundary-integral2} are well-defined. For example, in \eqref{eq-boundary-integral2},
	since $P_\Omega f$ has harmonic coefficients in $\Omega$, it has a trace of class $W^{-\frac{1}{2}}$ on $b\Omega_2$,
	and $g$ (which has $W^1$ coefficients) admits a trace of class $W^{\frac{1}{2}}$.
	\item  The perfect pairings~\eqref{eq-boundary-integral1} and \ref{eq-boundary-integral2}, or more generally the continuous bijection \eqref{eq-lambda-bijection}  are formal analogs
	in the theory  of  Sobolev Dolbeault cohomology of the classical \emph{Alexander duality}
	in topology (see \cite{bredon}).  See \cite{golovin} for analogous results for the cohomology of coherent analytic sheaves.
	\item For $n=1$, \eqref{eq-boundary-integral2} shows the following:
	the Bergman space $H^{0,0}_{L^2}(\Omega)$ of the annulus is the dual of the space $H^{0,0}_{W^1}(\Omega_2)$ of holomorphic functions
	of class $W^1$ on the hole, and the duality pairing is given by
	\[ f,g \mapsto \int_{b\Omega_2} f(z)g(z) dz,\]
	where the integral on the right is the usual line integral of one-variable complex analysis. The other pairing \eqref{eq-boundary-integral1} similarly identifies the dual
	of the Bergman space of the hole as the space of  $W^1$-holomorphic functions on the annulus by the same pairing.
	These  may be thought of as  analogs in the $L^2$-setting  of a classical result in complex analysis of one variable due to Grothendieck-Köthe-da Silva
	%	characterizing the dual of the space of holomorphic functions on a planar domain as  a space of functions on the complement via a boundary integral
	(see \cite[p. 67ff.]{rubel}).  These considerations can be generalized to higher dimensions to obtain an integral representation of functions
	in Bergman spaces in terms of their $W^{-\frac{1}{2}}$ boundary values.
\end{enumerate}
\begin{proof}[Proof of Proposition~\ref{prop-shaw}] For part (a), all that remains to be shown is that the pairing \eqref{eq-shaw1} can also be represented by \eqref{eq-boundary-integral1}.
	But
	\begin{align*}
	\int_{\Omega_2}\lambda^{p,q}\left([f]\right)\wedge [g]&= \int_{\Omega_2}\lambda^{p,q}\left([f]\right)\wedge [P_{\Omega_2}g]=
	\int_{\Omega_2}\dbar Ef\wedge P_{\Omega_2}g\\&=\int_{\Omega_2}\dbar (Ef\wedge P_{\Omega_2} g)= \int_{\Omega_2}d(Ef\wedge P_{\Omega_2}g)= \int_{b\Omega_2}f \wedge P_{\Omega_2}g.
	\end{align*}
	Part (b) is proved exactly the same way.
\end{proof}

\subsection{Splitting of the Sobolev cohomology of an annulus} The exact sequence  \eqref{eq-short-main} splits,  like any other
exact sequence of vector spaces  (see \cite[pp. 132 ff.]{lang}). Therefore, there is an injective linear map $\mu: H^{p,q+1}_{c,W^s}(\Omega_2)\to H^{p,q}_{W^{s+1}}(\Omega)$  such that we
have an algebraic direct sum decomposition of vector spaces:
	\begin{equation}
\label{eq-alg} 	H^{p,q}_{W^{s+1}}(\Omega)=  R^{p,q}_*(H^{p,q}_{W^{s+1}}(\Omega_1))\oplus \mu(H^{p,q+1}_{c,W^s}(\Omega_2)).
\end{equation}
This splitting is not \emph{topological}, i.e., the topology on $H^{p,q}_{W^{s+1}}(\Omega)$ is
not the direct sum topology of the two summands on the right hand side (given by the semi-inner-product \eqref{eq-directsum}.) Neither is this splitting \emph{natural}, i.e., $\mu$ is not
determined by the exact sequence \eqref{eq-short-main}. However, note that \eqref{eq-alg} already determines the cohomology of an annulus as a vector space, and gives a condition for its vanishing.

More information about the topology of $H^{p,q}_{W^{s+1}}(\Omega)$ can be obtained applying the observations of  Section~\ref{sec-shortexact}:

\begin{prop}\label{prop-splitting}
In the exact sequence \eqref{eq-short-main}:
	\begin{enumerate}
		\item  suppose that the cohomology $H^{p,q+1}_{c, W^s}(\Omega_2)$ is Hausdorff. Then, in the splitting \eqref{eq-alg}, the map $\mu$ can be so chosen that it is a linear homeomorphism
		onto its image, and the splitting is topological. However, the splitting is not natural.
		\item  if both the cohomology groups  $H^{p,q+1}_{c, W^s}(\Omega_2) $ and  $H^{p,q}_{W^{s+1}}(\Omega)$ are
		Hausdorff, then $H^{p,q}_{W^{s+1}}(\Omega_1)$ is also Hausdorff, and we have a natural orthogonal splitting
		\begin{equation}
		\label{eq-ortho} 	 	H^{p,q}_{W^{s+1}}(\Omega) = R^{p,q}_*(H^{p,q}_{W^{s+1}}(\Omega_1))\oplus (\lambda^{p,q})^\dagger (H^{p,q+1}_{c,W^s}(\Omega_2)),
		\end{equation}
		where $(\lambda^{p,q})^\dagger :H^{p,q+1}_{c,W^s}(\Omega_2)\to H^{p,q}_{W^{s+1}}(\Omega)$ is the Hilbert space adjoint of the map
		$\lambda^{p,q}: H^{p,q}_{W^{s+1}}(\Omega)\to H^{p,q+1}_{c,W^s}(\Omega_2)$ of \eqref{eq-lambdadef}. 	\end{enumerate}
\end{prop}

%%%%%%%%%%%%%%%%%%%%%%%%%%%%%%%%%%
\subsection{$L^2$-cohomology of an annulus with pseudoconvex hole}\label{sec-long}

We now combine the long exact sequence \eqref{eq-long-hole-min} and the short exact sequence \eqref{eq-short-main} to prove the following result, which gives $L^2$-estimates on an annulus provided we
have $L^2$-estimates on the envelope and pseudoconvex hole:
	
	\begin{thm} \label{thm-annulus1}
		Let $\Omega=\Omega_1\setminus \ol{\Omega}_2$ be an annulus in $\cx^n$  such that $H^{p,q}_{L^2}(\Omega_1)$ is Hausdorff in each degree, and
		$\Omega_2$ is pseudoconvex with $\mathcal{C}^{1,1}$ boundary.
%		Let $R_*: H^{p,q}_{L^2}(\Omega_1)\to H^{p,q}_{L^2}(\Omega)$ be the restriction map
%		on cohomology, induced by the restriction map on forms.
		 Then $H^{p,q}_{L^2}(\Omega)$ is Hausdorff for each $0\leq p,q\leq n$, and we have for $0\leq p \leq n$:
		
		\[ H^{p,q}_{L^2}(\Omega)= \begin{cases}  R_*^{p,q}(H^{p,q}_{L^2}(\Omega_1)) & \text{ if } 0\leq q \leq n-2\\
		R_*^{p,n-1}(H^{p,n-1}_{L^2}(\Omega_1))\oplus (\lambda^{p,n-1})^\dagger \left(H^{p,n}_{c,W^{-1}}(\Omega_2) \right)&\text{ if } q=n-1\\
		0& \text{ if } q=n,
		\end{cases}\]
		where the notation is as in Theorem~\ref{thm-short} and Proposition~\ref{prop-splitting}.
	\end{thm}

	\begin{proof} First we show that the groups $H^{p,q}_{L^2}(\Omega)$ are all Hausdorff. Thanks to Proposition~\ref{prop-serre-laufer}, it suffices to show that the groups
		$H^{p,q}_{c, L^2}(\Omega)$ are Hausdorff for all $0\leq q \leq n$.
		
		Since the groups $H^{p,q}_{L^2}(\Omega_1)$ are all Hausdorff,  Proposition~\ref{prop-serre-laufer}
		shows that all the groups $H^{p,q}_{c, L^2}(\Omega_1)$ are also Hausdorff.   Also, by Theorem~\ref{thm-hormander-kohn}, we have for the hole that $H^{p,q}_{W^1}(\Omega_2)=0$ when $q\geq 1$.
		
		Now, from \eqref{eq-long-hole-min}, we have the exact fragment of semi-inner-product spaces
		\[ H^{p,q-1}_{W^1}(\Omega_2)
		\xrightarrow{\ell^{p,q-1}}H^{p,q}_{c, L^2}(\Omega) \xrightarrow{(\eh_*)^{p,q}}H^{p,q}_{c, L^2}(\Omega_1) \xrightarrow{S^{p,q}}H^{p,q}_{W^1}(\Omega_2)\]
		which reduces, for $2 \leq q \leq n$ to
		\[ 0 \to H^{p,q}_{c, L^2}(\Omega) \xrightarrow{(\eh_*)^{p,q}}H^{p,q}_{c, L^2}(\Omega_1) \to 0.\]
		Since $H^{p,q}_{c, L^2}(\Omega_1)$ is Hausdorff and $(\eh_*)^{p,q}$ is a continuous injective map, it follows that
		$H^{p,q}_{c, L^2}(\Omega)$ is Hausdorff for $2\leq q \leq n$.
		
%		As a result, $(\eh_*)^{p,q}$ is a continuous bijection of Hilbert spaces, and therefore by the closed-graph theorem, an isomorphism of the Hilbert spaces (as topological vector spaces). The adjoint map $ ((\eh_*)^{p,q})': (H^{p,q}_{c, L^2}(\Omega_1))'\to (H^{p,q}_{c, L^2}(\Omega))'$
%		is also an isomorphism.  Since we have the Hausdorff property  in all degrees, the dual $(H^{p,q}_{c, L^2}(\Omega_1))'$ can be identified via Serre duality with $H^{n-p,n-q}_{L^2}(\Omega_1)$.
%		
		Applying  Proposition~\ref{prop-serre-laufer}	 (with $s=0$) to ${H^{n-p,n}_{L^2}(\Omega)}=0$ shows that $H^{p,1}_{c, L^2}(\Omega)$ is Hausdorff.
		
		Finally Part~\ref{part-6} of Proposition~\ref{prop-maxrealproperties} shows that $H^{p,0}_{c, L^2}(\Omega)=0$ and is therefore Hausdorff.
		
		Letting $s=-1$ in \eqref{eq-short-main}, we note that each term of the following exact sequence is Hausdorff for each $q$:
		\[0\to H^{p,q}_{L^2}(\Omega_1) \xrightarrow{{R}^{p,q}_*}H^{p,q}_{L^2}(\Omega)
			\xrightarrow{\lambda^{p,q}}H^{p,q+1}_{c,W^{-1}}(\Omega_2)\to 0, \]
			where $H^{p,q+1}_{c,W^{-1}}(\Omega_2)$ is Hausdorff from Corollary~\ref{cor-l2coh-compact}. Therefore, from the results of Section~\ref{sec-shortexact}, we conclude that
			there is an orthogonal direct sum decomposition of Hilbert spaces
			\[ H^{p,q}_{L^2}(\Omega)= {R}^{p,q}_*(H^{p,q}_{L^2}(\Omega_1)) \oplus (\lambda^{p,q})^\perp \left(H^{p,q+1}_{c,W^{-1}}(\Omega_2)\right)\]
			which gives the result on noting the value of $H^{p,q+1}_{c,W^{-1}}(\Omega_2)$ given by Corollary~\ref{cor-l2coh-compact}.
		\end{proof}

	\subsection{The $\dbar$-problem with mixed boundary conditions.}
	An important special case of the construction of Section~\ref{sec-mixed} is when the realization $\db$ on $\Omega_1$ is the usual maximal $W^s$-realization. In this case, we obtain a mixed realization
	$A^{p,*}_{(\db,c), W^s}(\Omega)$ which coincides with the maximal  $W^s$-realization along $b\Omega_1$ and with the minimal $W^s$-realization along $b\Omega_2$. We denote the corresponding
	cohomology groups by $H^{p,q}_{\mathrm{mix}, W^s}(\Omega)$. Theorem~\ref{thm-long1} now implies that the sequence of semi-inner-product spaces and continuous maps
		\begin{equation}\label{eq-long-hole-max} \cdots \xrightarrow{{S}^{p,q-1}} H^{p,q-1}_{W^{s+1}}(\Omega_2)
		\xrightarrow{\ell^{p,q-1}}H^{p,q}_{\mathrm{mix},W^s}(\Omega) \xrightarrow{(\eh_*)^{p,q}}H^{p,q}_{W^s}(\Omega_1) \xrightarrow{{S}^{p,q}}H^{p,q}_{W^{s+1}}(\Omega_2)
		\xrightarrow{\ell^{p,q}}\cdots\end{equation}
	is exact. The interested reader can easily deduce the  results obtained in \cite{lishaw} regarding solution of the $\dbar$-problem with mixed boundary conditions  from \eqref{eq-long-hole-max} by making the appropriate vanishing assumptions on the $L^2$-cohomologies of the hole and envelope.

\section{Solving $\dbar$  with prescribed support}\label{sec-prescribed}
\subsection{Application of the Hartogs phenomenon} Combining Hartogs phenomenon with the exact sequence \eqref{eq-short-main} leads to the following proposition and its
corollary. Special cases  of these were noted in \cite{meichichristine}:
\begin{prop}\label{prop-connected-complement}
	Let $D\subset\cx^n$ be a bounded domain such that $\cx^n\setminus D$ is connected. If $n\geq 2$ then for $0\leq p \leq n$ we have for each integer $s$ that
	\[ H^{p,1}_{c, W^s}(D)=0.\]
\end{prop}
\begin{proof}
 Let $\Omega_1$ be a large ball such that $D\Subset \Omega_1$. Denoting $D$ also by $\Omega_2$,  let $\Omega= \Omega_1\setminus \ol{\Omega_2}$ be the annulus of which
$\Omega_1$ is the envelope and $D=\Omega_2$ is the hole. Notice that $\Omega$ is connected, so by the Hartogs phenomenon, the map
$R^{p,0}_*: H^{p,0}_{W^{s+1}}(\Omega_1)\to H^{p,0}_{W^{s+1}}(\Omega)$ is an isomorphism, being simply the restriction map of holomorphic forms.
The exact sequence \eqref{eq-short-main} with $q=0$ reduces
to
\[0\to H^{p,0}_{W^{s+1}}(\Omega_1)\xrightarrow{\mathrm{isom}}H^{p,0}_{W^{s+1}}(\Omega)\xrightarrow{0}H^{p,1}_{c,W^s}(D)\to 0 \]
where the surjectivity of $R^{p,0}$ implies that $\lambda^{p,0}=0$.
Exactness at $H^{p,1}_{c, W^s}(D)$ implies that the zero map into this space is surjective, which gives us $H^{p,1}_{c, W^s}(D)=0$.
\end{proof}
\begin{cor}Under the above hypotheses, if $D$ is Lipschitz,
	$H^{p, n-1}_{W^s}(D)$ either vanishes, or is an infinite dimensional indiscrete (and therefore non-Hausdorff) space.
\end{cor}
\begin{proof}
 Since $H^{n-p,1}_{c, W^{-s}}(D)=0$, 	by Proposition~\ref{prop-serre-laufer},
	 we have that $\Red H^{p, n-1}_{W^s}(D)=0$.
\end{proof}
\subsection{Application of vanishing theorems} In this section, we will combine our algebraic approach with the following known vanishing result on annuli:
\begin{thm}[Shaw \cite{shawannulus1985}]\label{thm-mei-chi}
	{Let $\Omega=\Omega_1\setminus\ol{ \Omega_2}$ be a bounded annulus with $\mathcal{C}^k$ boundary, $k\geq 2$, determined by a pseudoconvex envelope $\Omega_1$
	and a pseudoconvex hole $\Omega_2$ in $\cx^n$, where $n\geq 2$.  Then for each $0<s\leq k-1$  we have the following
	\[ H^{p,q}_{W^s} (\Omega)= \begin{cases} \text{Hausdorff and infinite dimensional} & \text{if  } q=0\\
	0 & \text{if } 1\leq q \leq n-2\\
	\text{infinite dimensional} & \text{if } q=n-1\\
	0 & \text{if } q=n.
	\end{cases} \]}
\end{thm}
\begin{proof}
  When the boundary of $\Omega$ is smooth, this follows from \cite[ Theorems 3.1 and 3.2]{shawclosed}.  In particular, \cite[ Theorems 3.1]{shawclosed}  demonstrates that a solution operator continuous in $W^s(\Omega)$ exists on the orthogonal complement of the space of harmonic forms (see also \cite{shawannulus1985, hormander2004}), and
  \cite[ Theorem 3.2]{shawclosed}  demonstrates that the space of harmonic forms vanishes.  We note that Shaw makes use of the space of harmonic forms for the $\dbar$-Neumann operator with weights, but the dimension of this space is independent of the weight (see the final sentence of  \cite[Theorem 3.19]{kohn73}, for example).  In particular, the space of harmonic forms can be identified with the orthogonal complement of the range of $\dbar$ in $\ker\dbar$, and since neither of these spaces depends on the weight, the dimension of the orthogonal complement is also independent of the weight.

  For $k\geq 2$, this will follow from the induction procedure detailed in  \cite[Lemma 3 ]{harrington2009}.  Although the results in \cite{harrington2009} are only stated for pseudoconvex domains, the induction argument will work provided that the base case is proven, and this will follow from \cite{DebrajPhil}.
\end{proof}

\noindent\textbf{Remark:} It does not seem to be known whether $H^{p,n-1}_{W^s} (\Omega)$ is Hausdorff if $s\geq 1$. There is good reason to suspect that, as in the case $s=0$, this group is
Hausdorff, and it would be interesting to verify this claim. However, we will see below that conventional techniques based on the $\dbar$-Neumann problem
are unlikely to be successful in answering this question.

\begin{prop}\label{prop-prescribed1}
Let $D\subset\cx^n, n\geq 1$,
	 be a bounded pseudoconvex domain with $C^{s+2}$ boundary for $s\geq 0$, and let $0\leq p \leq n$. Then
\[
H^{p,q}_{c, W^s}(D)= \begin{cases}
0 & \text{ if }0\leq q \leq n-1,\\
\text{infinite dimensional} & \text{ if } q=n.
\end{cases}
\]

\end{prop}

\begin{rem}
  We note that for $q\geq 2$, this can not be proven using standard $\dbar$-Neumann techniques, as we will demonstrate in Section \ref{sec:canonical_solution_prescribed_support}.
\end{rem}

\begin{proof}
		 For $q=0$, this has already been proven in part~\ref{part-6} of Proposition~\ref{prop-maxrealproperties} above.

		 The case
$n\geq 2$ and $ q=1$ follows immediately from Proposition~\ref{prop-connected-complement}.

 For the remaining cases, let $\Omega_1$ be a large ball such that $D\Subset \Omega_1$. Denote $D$ by $\Omega_2$ also,  and let $\Omega= \Omega_1\setminus \ol{\Omega_2}$ be the  annulus determined by $\Omega_1$ and $\Omega_2$.  The exactness  of \eqref{eq-short-main} holds for this annulus.

For the case $n=q=1$, we use \eqref{eq-short-main} (with $q=0$) to obtain an isomorphism of vector spaces
\[ \frac{H^{p,0}_{W^{s+1}}(\Omega)}{R\left(H^{p,0}_{W^{s+1}}(\Omega_1)\right)}\cong H^{p,1}_{c,W^s}(D),\]
where $R$ is the restriction map on forms from $\Omega_1$ to $\Omega$.  The left hand side is  infinite dimensional since, for example, for each
$w\in D$, the images of  $z\mapsto (z-w)^{-k}$ with $k\geq 1$ form an infinite linearly independent family in it.  The result follows in
this case.

Now let $n\geq 2$ and $2\leq q\leq n-1$.
By Theorem~\ref{thm-mei-chi}, we have
	$ H^{p,q-1}_{W^{s+1}}(\Omega)=0,$
	and by  Theorem~\ref{thm-hormander-kohn}, we have that
	$H^{p,q-1}_{W^{s+1}}(\Omega_1)=0.$
	Therefore, from the exactness  of \eqref{eq-short-main} (replacing $q$ with $q-1$) at $H^{p,q}_{c,W^s}(D)$ we see that
	\[ H^{p,q}_{c,W^s}(D)=0 \text{ for }  2\leq q \leq n-1.\]
	
	For $q=n\geq 2$,  we use \eqref{eq-short-main} (with $q=n-1$) to obtain
	 	\[0\to H^{p,n-1}_{W^{s+1}}(\Omega) \xrightarrow{\lambda^{p,n-1}} H^{p,n}_{c,W^s}(D)\to 0.\]
	 	Here, we have again applied Theorem~\ref{thm-hormander-kohn}. Therefore, $\lambda^{p,n-1}$ is algebraically an isomorphism of
	 	vector spaces, and since, by Theorem~\ref{thm-mei-chi}, the space $H^{p,n-1}_{W^{s+1}}(\Omega)$ is infinite-dimensional,
	 	it follows that $H^{p,n}_{c,W^s}(D)$ is infinite-dimensional.

\end{proof}

We also have the following consequence of Proposition~\ref{prop-prescribed1}:
\begin{cor} \label{cor-negative1}
	Let $D\subset\cx^n, n\geq 1$,
	be a bounded pseudoconvex domain with $C^{s+2}$ boundary for $s\geq 0$, and let $0\leq p \leq n$. Then
	\[
	H^{p,q}_{W^{-s}}(D)= 0, \quad \text{for} \quad 2\leq q \leq n.
	\]
\end{cor}

\begin{proof}
	By  Proposition~\ref{prop-prescribed1}, $H^{p,q}_{c, W^s}(D)=0$ for $0\leq q \leq n-1$.
	From this we conclude that (i) $H^{n-p, n-q+1}_{W^{-s}}(D)$ is Hausdorff if $0\leq q \leq n-1$,
	 by part (1) of Proposition~\ref{prop-serre-laufer}, and that  (ii) $H^{p,q-1}_{c, W^s}(D)=0$ if $1\leq q\leq n$, by a simple change in index.
	 Applying part (2) of Proposition~\ref{prop-serre-laufer} to the statements (i) and (ii)  we conclude that $H^{n-p, n-q+1}_{W^{-s}}(D)=0$ if $1\leq q \leq n-1$, which on renaming the indices is precisely the claim of the corollary.

	Notice that we cannot
	say anything about the case $q=1$, since we do not have information about whether $H^{p,n-1}_{W^s} (\Omega)$ is Hausdorff if $s\geq 1$.
\end{proof}
We now solve the $\dbar$-problem with prescribed support in an annulus:

\begin{prop}\label{prop-prescribed2}
	Let $\Omega=\Omega_1\setminus \ol{\Omega}_2\subset \cx^ n$ be an annulus, and $s\geq 0$. Assume that $\Omega_1$ and $\Omega_2$
	are pseudoconvex with smooth boundaries.  Then
	\[ H^{p,q}_{c,W^s}(\Omega)=\begin{cases}0 &\text{ if } q=0\\
	\text{infinite dimensional}  & \text{ if } q=1\\
0 & \text{ if } 2\leq	 q \leq n-1 \\
\text{infinite dimensional}&\text{ if } q=n.
	\end{cases}\]

\end{prop}
\begin{proof}  The case $q=0$ has been established in Part (6) of Proposition~\ref{prop-maxrealproperties} above.
	
	For the case $q=1$, we consider the fragment of the exact sequence \eqref{eq-long-hole-min}
	\[H^{p,0}_{c,W^s}(\Omega_1) \xrightarrow{{S}^{p,0}}H^{p,0}_{W^{s+1}}(\Omega_2)\xrightarrow{\ell^{p,0}}H^{p,1}_{c,W^s}(\Omega) \xrightarrow{(\eh_*)^{p,1}}H^{p,1}_{c,W^s}(\Omega_1),\]
	where we have $H^{p,0}_{c,W^s}(\Omega_1)=H^{p,1}_{c,W^s}(\Omega_1)=0$ by Proposition~\ref{prop-prescribed1} above. Therefore the
	following sequence is exact:
	\[0 \to H^{p,0}_{W^{s+1}}(\Omega_2)\xrightarrow{\ell^{p,0}}H^{p,1}_{c,W^s}(\Omega)\to 0,\]
so that $\ell^{p,0}$ is a linear isomorphism. But $ H^{p,0}_{W^{s+1}}(\Omega_2)$ is  infinite dimensional which implies that so is $H^{p,1}_{c,W^s}(\Omega)$.

	For $2\leq q \leq n-1$, we have that {$H^{p,q-1}_{W^{s+1}}(\Omega_2)=0$} by {Theorem~\ref{thm-hormander-kohn}} and $H^{p,q}_{c,W^s}(\Omega_1)=0$
	by Proposition~\ref{prop-prescribed1}. Therefore, by the exactness of \eqref{eq-long-hole-min} at $H^{p,q}_{c,W^s}(\Omega)$ we see that
	the fragment $0\to H^{p,q}_{c,W^s}(\Omega)\to 0$ is exact, which means that $H^{p,q}_{c,W^s}(\Omega)=0$.
	
	For $q=n$, we have that the following is exact from \eqref{eq-long-hole-min}
	\[ \to H^{p,n-1}_{W^{s+1}}(\Omega_2)
	\xrightarrow{\ell^{p,n-1}}H^{p,n}_{c,W^s}(\Omega) \xrightarrow{(\eh_*)^{p,n}}H^{p,n}_{c,W^s}(\Omega_1) \xrightarrow{{S}^{p,n}}H^{p,n}_{W^{s+1}}(\Omega_2)\to.\]
Since $H^{p,n-1}_{W^{s+1}}(\Omega_2)=H^{p,n}_{W^{s+1}}(\Omega_2)=0$ by  {Theorem~\ref{thm-hormander-kohn}}, this reduces to
\[0\to H^{p,n}_{c,W^s}(\Omega) \xrightarrow{(\eh_*)^{p,n}}H^{p,n}_{c,W^s}(\Omega_1)\to 0, \]
which shows that $(\eh_*)^{p,n}$ is a continuous linear isomorphism of semi-Hilbert spaces. Since $H^{p,n}_{c,W^s}(\Omega_1)$ is infinite dimensional, it follows that $H^{p,n}_{c,W^s}(\Omega) $ is also infinite dimensional.
\end{proof}

Using an argument similar to that used for Corollary~\ref{cor-negative1} we can prove the following:
	
\begin{cor}\label{cor-negative2}
		Let $\Omega=\Omega_1\setminus \ol{\Omega}_2\subset \cx^ n$ be an annulus where {$n\geq 4$} and $s\geq 0$. Assume that $\Omega_1$ and $\Omega_2$
	are pseudoconvex with smooth boundaries.  Then
	\[   H^{p,q}_{W^{-s}}(\Omega)=0, \quad \text{ for } 2\leq q \leq n-2.\]
\end{cor}

\subsection{Lack of continuity of the canonical solution in spaces with prescribed support}
\label{sec:canonical_solution_prescribed_support}
As shown in \cite[Theorem 9.1.2]{chenshaw}, the $\dbar$-Neumann operator and the Hodge star operator can be used to construct a solution to $\dbar u=f$ with prescribed support.  In particular, if $N_{p,q}^D$ denotes the $\dbar$-Neumann operator for some bounded pseudoconvex domain $D\subset\mathbb{C}^n$ and $\star:\Lambda^{p,q}\rightarrow\Lambda^{n-p,n-q}$ denotes the Hodge star operator, then whenever $f\in Z^{p,q}_{c,L^2}(D)$ for $1\leq q\leq n-1$ and $0\leq p\leq n$, the form
 \[ u=-\star\dbar N^D_{n-p,n-q}\star f\in A_{c,L^2}^{p,q-1}(D)\]
  solves $\dbar u=f$.  This suffices to prove Proposition \ref{prop-prescribed1} when $s=0$ and $1\leq q\leq n-1$.

When $q=1$, $u=-\star\dbar N^D_{n-p,n-1}\star f$ is the unique solution in $A^{p,0}_{c,L^2}(D)$ to $\dbar u=f$ by part~\ref{part-6} of Proposition \ref{prop-maxrealproperties}.  Hence, Proposition \ref{prop-prescribed1} implies that this $u$ must also lie in $A^{p,0}_{c,W^s}(D)$ whenever $f\in Z^{p,1}_{c,W^s}(D)$.  The situation is quite different when $q\geq 2$, and we can show that this operator does not suffice to solve $\dbar$ in $A^{p,q}_{c,W^s}(D)$ for $s\geq 2$ and $q\geq 2$.
\begin{prop}\label{prop-notcts}
  Let $D\subset\mathbb{C}^n$ be a smooth, bounded, pseudoconvex domain.  For any $0\leq p\leq n$ and $2\leq q\leq n$, there exists an infinite dimensional family of forms
   $f\in \mathcal{D}^{p,q}(D)$ supported in $D$ such that $\dbar f=0$ but $u=-\star\dbar N^D_{p,q}\star f$ fails to be in $A^{p,q-1}_{c,W^2}(D)$.
\end{prop}

\begin{proof}
  Let $B_1$ and $B_2$ be balls satisfying $\overline{B_1}\subset D$ and $\overline{B_2}\subset B_1$.  Let $\Omega=D\backslash\overline{B_2}$.  By Theorem 3.5 in \cite{shawclosed}, the $\dbar$-Neumann operator $N_{n-p,n-q+1}^\Omega$ exists for $0\leq p\leq n$ and $2\leq q\leq n$.

  Let $\rho_2$ be a smooth defining function for $B_2$ and fix $w\in bB_2$ at which $\dbar\rho_2(w)=\frac{\partial\rho_2}{\partial\bar z_1}(w)d\bar z_1$. Let $\rho$ be a smooth defining function for $D$ such that $\dbar\rho(w)=d\bar z_2$.  Within this proof, we use $\dbar^*_D$ and $\dbar^*_\Omega$ to distinguish the adjoint of $\dbar$ and its domain on each of these domains.  Define $g\in C^\infty_{n-p,n-q+1}(\overline D)$ by
  \[
    g=\dbar^*_D(\rho dz_1\wedge\ldots\wedge dz_{n-p}\wedge d\bar z_1\wedge\ldots\wedge d\bar z_{n-q+2}).
  \]
  At our fixed point $w$,
  \[
    g(w)=\left(\frac{\partial\rho_2}{\partial\bar z_1}(w)\right)^{-1}dz_1\wedge\ldots\wedge dz_{n-p}\wedge \dbar\rho_2\wedge d\bar z_3\wedge\ldots\wedge d\bar z_{n-q+2}\neq 0,
  \]
  when $n-1\geq q$, and
  \[
    g(w)=\left(\frac{\partial\rho_2}{\partial\bar z_1}(w)\right)^{-1}dz_1\wedge\ldots\wedge dz_{n-p}\wedge \dbar\rho_2\neq 0,
  \]
  when $q=n-1$.  In either case, we see that $g|_\Omega\notin\dom\dbar^*_\Omega$.  By Theorem 3.5 in \cite{shawclosed} again,
  \[
    g|_\Omega=\dbar\dbar^*_\Omega N_{n-p,n-q+1}^\Omega g+\dbar^*_\Omega\dbar N_{n-p,n-q+1}^\Omega g
  \]
  for $3\leq q\leq n$ and
  \[
    g|_\Omega=\dbar\dbar^*_\Omega N_{n-p,n-1}^\Omega g+\dbar^*_\Omega\dbar N_{n-p,n-1}^\Omega g+P_{n-p,n-1}^\Omega g,
  \]
  for $q=2$, where $P_{n-p,n-1}^\Omega$ is the orthogonal projection onto the infinite-dimensional kernel of $\Box^\Omega$ in $L^2_{n-p,n-1}(\Omega)$.  Since $g|_\Omega\notin\dom\dbar^*_\Omega$, $\dbar\dbar^*_\Omega N_{n-p,n-q+1}^\Omega g$ is nontrivial in either case.  We may assume that $B_1$ has been chosen to be sufficiently small so that $\dbar\dbar^*_\Omega N_{n-p,n-q+1}^\Omega g$ is nontrivial on $\Omega\backslash\overline{B_1}$.

  Let $\chi\in C^\infty_0(B_1)$ satisfy $\chi\equiv 1$ in a neighborhood of $\overline{B_2}$, and set
  \[
    h=\dbar((1-\chi)\dbar^*_\Omega N_{n-p,n-q+1}^\Omega g)
  \]
  on $\Omega$ and $h=0$ on $\overline{B_2}$.  By the interior regularity for the $\dbar$-Neumann problem, $h\in C^\infty_{n-p,n-q+1}(D)\cap L^2_{n-p,n-q+1}(D)$.
  Since $g\in\dom\dbar^*_D$, we also have $h\in\dom\dbar^*_D$, and on $D\backslash\overline{B_1}$ we have
  \[
    \dbar^*_D h=\dbar^*_D\dbar\dbar^*_\Omega N_{n-p,n-q+1}^\Omega g=\dbar^*_D g=0.
  \]
  Since $\dbar h=0$, we conclude that $h$ has harmonic coefficients on $D\backslash\overline{B_1}$, and hence the Hopf Lemma guarantees that the gradients of these coefficients are non-vanishing on $b D$ (if these coefficients are constant on $b D$, then the $\dbar$-Neumann boundary condition requires $h=0$ on $D\backslash\overline{B_1}$, contradicting the non-triviality of $h$ on this set).  Hence, $h\notin A^{n-p,n-q+1}_{c,W^2}(D)$.

  Now, since $h\in\range\dbar$ on a pseudoconvex domain, $h=\dbar N_{n-p,n-q}^D\dbar^*_D h$.  Using Lemma 9.1.1 in \cite{chenshaw},
  \[
    \star h=-\star\dbar N_{n-p,n-q}^D\star(\star\dbar^*_D\star(\star h))=\star\dbar N_{n-p,n-q}^D\star\dbar(\star h).
  \]
  Let $f=\dbar(\star h)$.  Then on $D\backslash\overline{B_1}$, $f=-(-1)^{p+q}\star\dbar^*_D h=0$, so $f$ is supported in $\overline{B_1}$.  Since $h$ is smooth in the interior of $D$, we conclude that $f\in C^\infty_{0,(p,q)}(D)$.  However, $u=-\star\dbar N_{n-p,n-q}^D\star f=-\star h$, so $u\notin A^{p,q-1}_{c,W^2}(D)$.

  Since we can cover $D$ with a countable collection of disjoint closed balls, we can apply this construction to each ball in the family to obtain an infinite dimensional family of forms $f$.

\end{proof}

\numberwithin{equation}{section}
\renewcommand{\theequation}{\thesection.\arabic{equation}}
\renewcommand{\thesection}{\Alph{section}}
\setcounter{section}{0}
\setcounter{subsection}{0}
\section{Appendix:  Non-Hausdorff functional analysis}\label{sec-appendix}

In this appendix  we collect together definitions and results about the non-Hausdorff topological vector spaces that
arise in the study of cohomology groups.  Much of this material consists of  routine variations on well-known results.
Proofs are included only when they are sufficiently different from the classical situation.

\subsection{Semi-inner-product spaces }\label{sec-sip}
\subsubsection{Definitions and basic properties}
By a \emph{semi-normed (linear) space}, we mean a complex vector space $V$ along with a distinguished seminorm $\norm{\cdot}:V\to \rl$  (see \cite[Definition~7.3, page 59]{treves}). The semi-normed space $(V, \norm{\cdot})$ a  becomes a (not necessarily Hausdorff) topological vector space under the \emph{natural semi-metric}  $ d(x,y)= \norm{x-y}$.  Using the semi-metric, one defines a topology on $V$: a basis for the topology consists of
the semi-balls
\[ B(x,\epsilon)=\{y\in X\mid  d(x,y)<\epsilon\}.\]
However, the topology so obtained is not necessarily Hausdorff. For example, the closure of the singleton $\{x\}$ consists of all $y\in V$ such that
$d(x,y)=0$, which may be true for a point different from $x$.

The  semi-normed space
 $(V, \norm{\cdot})$  is a \emph{semi-inner-product} space (\emph{SIP space} for short) if there is a \emph{semi-inner-product} $(\cdot, \cdot)$  on $V$ such that
$\norm{x}^2=(x, x)$ for each $x\in V$, where   a semi-inner product is a sesquilinear, hermitian-symmetric form on $V$ which satisfies $(x,x)\geq 0$. (It
is not assumed that $(x,x)=0$ only if $x=0$). It is not difficult to check that the Cauchy-Schwarz inequality and the triangle inequality continue to hold in a SIP space (though the conditions under which equality holds in these inequalities need to be modified). By the Cauchy-Schwarz inequality, we have the following:
If $X$ is a SIP space and  $z\in X$ is such that $\norm{z}=0$, then for all
		$x\in X$, we have $\ipr{x,z}=0.$

We say that a semi-normed space $V$  is complete  if each Cauchy sequence converges, i.e., if $\{w_j\}\subset V$ is such
that $\lim_{j,k\to \infty} \norm{w_j-w_k}=0$ then there is $w\in V$ such that $w_j\to w$. A SIP space which is complete in its semi-norm is
said to be a \emph{semi-Hilbert space}.

%The proof of the following characterization of SIP spaces
%is identical to the analogous statement for inner-product spaces (\cite{jordanvonneumann}):
%\begin{prop} \label{prop-jordan}  A semi-normed space $(V, \norm{\cdot})$ is a SIP space if and only if for all $x,y\in V$ we have the parallelogram identity:
%	\begin{equation}
%	\label{eq-parallelogram}
%	\norm{x+y}^2 + \norm{x-y}^2 = 2 \norm{x}^2 + 2 \norm{y}^2.
%	\end{equation}
%\end{prop}
The following elementary result characterizing continuous maps is proved in the same way as the Hausdorff case:
\begin{prop}\label{prop-continuity}
	Let  $T:X\to Y$ be a linear map of SIP spaces.  Then $T$ is continuous if and only if there is a $C\geq 0$ such that $\norm{Tx}_Y \leq C \norm{x}_X$.
\end{prop}
Given two SIP spaces $X$ and $Y$, there is a natural semi-inner product on the algebraic direct sum $X\oplus Y$, given by
\begin{equation}
\label{eq-directsum}
 \ipr{(x,y),(x',y')}_{X\oplus Y}= \ipr{x,x'}_X + \ipr{y,y'}_Y. \end{equation}
Notice, however, that if $Z=X\oplus Y$ is a direct sum of SIP spaces, then we do not have in general that $Y=X^\perp=\{z\in Z : \ipr{z,x}=0 \text{ for all } x\in X\}$. Indeed
if $X$ is indiscrete, then $X^\perp=Z$.

The following  proposition, versions of which   are widely known,
describes the structure of a general SIP space as the sum of an indiscrete and a Hausdorff part. The indiscrete part $\Ind(X)$ and the reduced form
$\Red(X)$ of a semi-inner-product space are defined as in \eqref{eq-indiscrete} and \eqref{eq-reduced}.   We establish the linear homeomorphism \eqref{eq-direct}:

\begin{prop}\label{prop-structure} Let $(X, \ipr{\cdot,\cdot}	)$ be a SIP space.
	Then $\Ind(X)$ is the unique indiscrete  closed linear subspace of $X$
	such that if $Z$ is any linear subspace of $X$ algebraically complementary to $\Ind(X)$ then
	\begin{enumerate}
		\item restricted to  $Z$, the semi-inner product of $X$ is an inner product, so $Z$ is Hausdorff.
		\item 	the semi-inner product space $\Ind(X)\oplus Z$, thought of as a topological vector space,
		is linearly homeomorphic to $X$.
	\end{enumerate}
\end{prop}
\begin{proof}
	Recall that $\Ind(X) =\{x\in X: \norm{x}=0\}$ by \eqref{eq-indiscrete}. Then $\Ind(X)$ is a  closed linear subspace of $X$ and with the restricted semi-norm, $\Ind(X)$ is indiscrete.  Let $Z$ be a  linear subspace of $X$ which is algebraically complementary to $\Ind X$, i.e.,
	each element of $X$ can be uniquely written as $y+z$ where $y\in \Ind(X)$ and $z\in Z$.
	
	If $z\in Z$ is such that $\norm{z}=0$, then $z\in \Ind(X)$. But since $\Ind(X)\cap Z=\{0\}$ it follows that $z=0$. Consequently, $Z$ is
	an inner-product space.

	Let the map $P:X\to \Ind(X)$ be
	defined by $Px=y$, where $x=y+z$ is the unique representation of $x$ as the sum of an element $y\in \Ind(X)$ and $z\in Z$, so that $P$ is the projection map  onto $\Ind(X)$  and has kernel $Z$.
	Then $P$ is continuous, since $\Ind X$ has the indiscrete topology. The map
	\[ X \to \Ind(X) \oplus Z\]
	given by
	\[ x \mapsto (Px, x-Px)\]
	is clearly injective, and it is  continuous with respect to the natural topology on $\Ind(X) \oplus Z$, since each of the two components is continuous.  The map $\Ind(X)\oplus Z \to X$  given by $(y,z)\mapsto y+z$ is its continuous
	linear inverse, so this is a linear homeomorphism.  The result follows.
\end{proof}

\subsubsection{Quotients} \label{sec-quotient}
Let $X$ be a TVS (topological vector space) and $Z\subset X$ a linear subspace of $X$.  The  pair $(X/Z, \pi)$ where
$X/Z$ is the quotient topological vector space and $\pi:X\to X/Z$ is the continuous natural projection enjoys the following universal property:
if $f:X\to W$ is a continuous linear map of topological vector spaces such that $f|_Z=0$, then there is a unique continuous linear map
$\ol{f}: X/Z\to W$ such that $f=\ol{f}\circ \pi$, i.e. the following diagram commutes:
\begin{equation}\label{diag-quotient}
	\begin{tikzcd}
X \arrow[r, "f"] \arrow [d, "\pi"] & W\\
X/Z \arrow[ru, dashrightarrow, "\ol{f}" ']&
\end{tikzcd}\end{equation}
Notice that it follows from this property that  as a linear space, $X/Z$ can be identified with
the  quotient vector space, defined in the  usual way (see \cite[p.~15]{treves}). The following proposition shows that SIP spaces, unlike inner-product spaces, are stable with respect to quotients.
\begin{prop}\label{prop-quotient}
	Let $(X,(\cdot,\cdot))$ be an SIP space, and let $Z\subset X$ be a linear subspace.  There is a
	semi-inner product $\langle\cdot, \cdot\rangle$ on the
	quotient vector space $X/Z$ such that 	the semi-norm topology on $(X/Z, \langle\cdot,\cdot\rangle)$ coincides with the quotient topology.
	%\begin{enumerate}
	%	\item the semi-norm topology on $(X/Y, \langle\cdot,\cdot\rangle)$ coincides with the quotient topology.
	%	\item
	%\end{enumerate}
\end{prop}
\begin{proof}
	Let $\pi:X\to X/Z$ be the quotient map. We define
	\begin{equation}
	\label{eq-quotientnorm}
	\norm{\pi(x)}_{X/Z}= \inf_{z\in Z} \norm{x+z},
	\end{equation}
	which is defined on all of $X/Z$ since $\pi$ is surjective, and is easily seen to be a seminorm using standard arguments.
	Let $f:X\to W$ be a continuous linear map of topological vector spaces such that $f|_Z=0$, and let $\ol{f}:X/Z\to W$ be the induced map.
	Now notice that for each $x\in X$ and $z\in Z$ we have for a $C\geq 0$ depending only on the map $f$:
	\[\norm{\ol{f}(\pi(x))}_W= \norm{f(x)}_W = \norm{f(x+z)}_W \leq C \norm{x+z}_X.\]
	Since this holds for each $z\in Z$, it follows that
	\[ \norm{\ol{f}(\pi(x))}_W \leq C \inf_{z\in Z}\norm{x+z}_X= C\norm{\pi(x)}_{X/Z},\]
	which shows that $\ol{f}$ is continuous, and therefore by the universal property \eqref{diag-quotient},
 the seminorm on $X/Z$
	generates the quotient topology. To complete the proof we only need to show that this semi-norm is generated by a semi-inner-product. A classic argument well-known for norms and inner products (see (\cite{jordanvonneumann})) shows that  a semi-norm  $\norm{\cdot}_V$ on a vector space $V$  is generated by  a semi-inner product if and only if for all $x,y\in V$ we have the \emph{parallelogram identity}:
	\begin{equation}
	\label{eq-parallelogram}
	\norm{x+y}^2_V + \norm{x-y}^2_V = 2 \norm{x}^2_V + 2 \norm{y}^2_V.
	\end{equation}
	Now, let $x,y\in X$ and $z,w\in Z$, and apply \eqref{eq-parallelogram} to $x+z$ and $ y+w$ in $X$. Then
	\[ \norm{x+y+(z+w)}^2 +\norm{x-y+(z-w)}^2 = 2 \norm{x+z}^2 + 2 \norm{y+w}^2.\]
	Let $u=z+w$ and $v=z-w$. Then, as $z$ and $w$ range independently over $Z$, so do $u$ and $v$. Therefore, we conclude by taking infima of both sides of the above equation that
	\[ \norm{\pi(x)+\pi(y)}_{X/Z}^2 + \norm{\pi(x)-\pi(y)}_{X/Z}^2 = 2\norm{\pi(x)}_{X/Z}^2 +2\norm{\pi(y)}_{X/Z}^2.\]
	Therefore, since the parallelogram identity holds for the seminorm,  the quotient $X/Z$ is a semi-inner-product space.
\end{proof}

The following fact is proved in the same way as in the  classical situation (cf. \cite[Proposition~4.5, page 34]{treves}):
\begin{prop} \label{prop-hausdorff-q}Let $X$ be a SIP space and $Z$ a linear subspace. Then the semi-inner-product on  $X/Z$ is an inner-product  (equivalently,
	$X/Z$ is Hausdorff) if and only if $Z$ is closed in $X$.
\end{prop}
In view of Proposition~\ref{prop-structure}, {for a semi-inner-product space $X$, the indiscrete part of $X$ is easily verified to be a closed subspace of $X$, and each closed subspace of $X$ contains the subspace $\Ind(X)$.  The reduced form of $X$ defined by \eqref{eq-reduced} is a \emph{normed}  (and therefore Hausdorff) space, thanks to Proposition~\ref{prop-hausdorff-q}.}
	\subsection{Short exact sequences of  semi-inner-product spaces}\label{sec-shortexact}
	Let
	\begin{equation}
	\label{eq-short-exact}
	 0\to X\xrightarrow{\rho}Y
		\xrightarrow{\lambda}Z\to 0
	\end{equation}
	be a short exact sequence of semi-inner-product spaces and continuous maps. We collect here a few simple observations about such sequences:
	
	\begin{enumerate}[wide]
		\item As with any exact sequence of vector spaces, we have an algebraic isomorphism of vector spaces $ Y/\rho(X)\cong Z$, and an algebraic splitting, i.e., a direct sum representation of vector spaces
		\begin{equation}\label{eq-splitting1}
		Y =  \rho(X) \oplus \mu(Z),
		\end{equation}
where $\mu:Z\to Y$ is an injective linear map. Notice that the splitting is not natural, i.e., the map $\mu$ is not determined by the exact sequence \eqref{eq-short-exact}. We emphasize that
this splitting is not topological, i.e, the topology on $Y$ may be different from the direct sum topology from the subspaces $\rho(X)$ and $\mu(Z)$.
		
		\item Since $\lambda:Y\to Z$ is continuous and vanishes on $\rho(X)$ by the
		universal  property of quotient TVS (see Section~\ref{sec-quotient} above), we obtain an induced continuous linear bijection
		\begin{equation}\label{eq-lambdabar}
		\ol{\lambda}: (Y/\rho(X))^{\mathrm{top}} \to Z,
		\end{equation}
		where $ (Y/\rho(X))^{\mathrm{top}}$ is the vector space $Y/\rho(X)$ endowed with the quotient topology.
		
		\item \emph{Assume now that in \eqref{eq-short-exact} the space $Z$ is Hausdorff.}
		
		Then since $\ol{\lambda}$ is continuous and injective, it follows that $(Y/\rho(X))^{\mathrm{top}} $ is Hausdorff, and by the closed-graph theorem, $\ol{\lambda}$ is an isomorphism.

		Since  $(Y/\rho(X))^{\mathrm{top}} $ is Hausdorff, Proposition~\ref{prop-hausdorff-q} implies that $\rho(X)$ is closed in $Y$, and therefore must contain the indiscrete subspace $\Ind(Y)$ of $ Y$
		(see \eqref{eq-indiscrete}).  Thanks to Proposition~\ref{prop-structure}, we have a
		direct sum decomposition of TVS
		\[\rho(X)= \Ind( Y)\oplus V, \]
		where $V$ is a linear subspace of $\rho(X)$ algebraically complementary to $\Ind(Y)$, i.e., each $x\in \rho(X) $ has a unique representation $x=y+z$, where $y\in  \Ind( Y)$ and $z\in V$. Let $W$ be an algebraic complement of $ \Ind( Y)$ in
		$ Y$ such that $V\subset W$ (such a complement exists for algebraic reasons). Then clearly $W$ is a Hilbert space in the inner product induced from $Y$ and
		$V$ is a closed subspace of $W$, so we have an orthogonal decomposition $W=V\oplus V'$, where $V'$ is the orthogonal complement of $V$ in $W$. Then we have a direct sum decomposition
		of TVS
		\[ Y= \rho(X) \oplus V'.\]
		It follows by exactness of \eqref{eq-short-exact} that the restriction
		$ \lambda|_{V'}:V'\to Z $
		is a bijective continuous linear map of Hilbert spaces, and therefore an isomorphism in the category of TVS by the closed-graph theorem. If we now let $\mu:Z\to Y$ be the inverse of $ \lambda|_{V'}$,
		we again obtain a splitting as in \eqref{eq-splitting1}, but  (a) now the splitting is topological, i.e., the topology on $Y$ is the same as the direct sum topology from the right hand side, (b) though the splitting  is still not natural, since the algebraic complement $W$ cannot be chosen naturally.
		
				\item \emph{Assume now that in \eqref{eq-short-exact} both  the spaces $Y$ and $Z$ are Hausdorff.}

		Since $\rho$ is continuous and injective and $Y$ is Hausdorff, it follows that $X$ is also Hausdorff. Therefore, we have, by standard results in the theory of Hilbert spaces, an orthogonal direct sum representation
		\[ Y= \img \rho \oplus (\img \rho)^\perp
		=  \img \rho \oplus (\ker \lambda)^\perp= \img \rho \oplus \img (\lambda^\dagger),\]
		where $\lambda^\dagger:Z\to Y$ is the Hilbert-space adjoint of $\lambda$. Therefore we have a natural orthogonal splitting
		\begin{equation}
		\label{eq-splitting2}
		Y= \rho(X)\oplus \lambda^\dagger(Z).
		\end{equation}

			\end{enumerate}
\subsection{Cochain complexes} In this paper we consider cochain complexes of the form
\[  \rightarrow E^{q-1}\xrightarrow{d^{q-1}} E^q \xrightarrow{d^q} E^{q+1} \xrightarrow{d^{q+1}}\]
where each $E^q$ is  an inner-product space and the differentials $d^q$ are continuous linear maps satisfying $  d^q\circ d^{q-1}=0$ for each $q$.
Note that it is not assumed that the space $E^q$ is complete in the norm induced by the inner product.
It is of course possible to consider much more general topologized cochain complexes (see \cite{meichichristine}).
The {\em cohomology groups} of the complex are the quotient vector spaces
\begin{equation}  \label{eq-cohomology} H^{q}(E) = \frac{Z^{q}(E)}{B^q(E)},\end{equation}
where $ Z^{q}(E)= \ker d^q $ and $ B^q(E)= \img d^{q-1}$
are the spaces of cocycles and { coboundaries} respectively.  Since $d^q$ is continuous,  $Z^q(E)$ is a closed subspace
of $E^q$. The following is clear:
\begin{prop}\label{prop-sip-cohomology}
	If $E$ is a cochain complex of inner-product spaces, then the cohomology group $H^q(E)$ of \eqref{eq-cohomology} has a natural structure of a semi-inner product space, and this semi-inner-product gives rise to the quotient topology. This semi-inner product space is an inner product (and therefore Hausdorff) space if and only if $B^q(E)$ is closed as a subspace of $Z^q(E)$. Further, the inner product space $H^q(E)$  is a Hilbert space provided that the space $E^q$ is a Hilbert space.
\end{prop}

Let $(E,d)$ and $(F,\delta)$ be cochain complexes of inner-product spaces. A continuous cochain map $f$ is given by continuous linear maps
$f^q: E^q\to F^q$ such that for each $q$ we have $ \delta^{q}\circ f^q =  f^{q+1}\circ d^q$, i.e. the following diagram commutes for each $q$:
\[
\begin{tikzcd}
E^q \arrow[r, "f^q"] \arrow[d, "d^q"]
& F^{q} \arrow[d, "\delta^q"] \\
E^{q+1} \arrow[r, "f^{q+1}"]
& F^{q+1}
\end{tikzcd}
\]
It is well-known (see \cite[Chapter XX]{lang}) that such a map induces a linear map of the cohomologies in each degree.
The functoriality of cohomology interacts nicely with the topology:
\begin{prop}\label{prop-ind-cont}
	Let $f: (E,d)\to (F,\delta)$  be a continuous cochain map between   cochain complexes of inner-product spaces. Then the induced map
	$ f^q_*:H^q(E)\to H^q(F)$
	is continuous for each $q$.
\end{prop}
\begin{proof}
Consider the following  diagram, where $\pi_E$ and $\pi_F$ denote the natural continuous projections onto the cohomology groups, and  which commutes by the definition of the induced map $f^q_*$:
	\[
	\begin{tikzcd}
	Z^q(E)\arrow[r,"f^q"] \arrow[d, "\pi_E"] \arrow[dr, dashed, "\pi_F\circ f^q"]
	& Z^q(F) \arrow[d,"\pi_F"]
	\\
	H^q(E)
	\arrow[r,swap,"f^q_*"]
	& H^q(F).
	\end{tikzcd}
	\]
	Using the universal property of quotients as in diagram \eqref{diag-quotient}, since $\pi_F\circ f^q$ is a continuous map from $Z^q(E)$ to $H^q(F)$ which vanishes on $B^q(E)$ it follows that the induced map $f^q_*$ is continuous. \end{proof}

%%%%%%%%%%%%%%%%%%%%
\bibliographystyle{alpha}
\bibliography{alexander}

\end{document}